%% file: FI_P2_TR.tex
\documentclass[12pt,draftcls,onecolumn]{IEEEtran}
\usepackage{hslTR}
\usepackage{fullpage}
\newenvironment{proof}
{\par\noindent\textbf{Proof}}{\eop\smallskip\vskip 3 pt}
\newcommand{\eop}
{\hspace*{\fill}{
		$\vcenter{\hrule height1pt
		\hbox{\vrule width1pt height5pt
		\kern5pt \vrule width1pt} \hrule height1pt}$} }
\input{Packages}

\newcommand{\ItemSp}{5mm}
\newcommand{\ItemLong}{8mm}
\newcommand{\figsize}{0.6\textwidth} 
\input{GenMacros}

\usepackage{rgsEnvironments}
\usepackage{rgsMacros}

\graphicspath{{./Figures/}}
\newboolean{Conference}
\setboolean{Conference}{false}
\newcommand{\NotConf}[1]{\ifthenelse{\boolean{Conference}}{}{#1}}
\newcommand{\IfConf}[2]{\ifthenelse{\boolean{Conference}}{#1}{#2}}   

\begin{document}
	\ititle{Forward Invariance of Sets for \\ Hybrid Dynamical Systems (Part II)}
	\iauthor{
	Jun Chai\\
	{\normalsize jchai3@ucsc.edu}\\
	Ricardo Sanfelice\\
	{\normalsize ricardo@ucsc.edu}}
	\idate{\today{}}
	\iyear{2020}
	\irefnr{0XX}
	\makeititle
	\tableofcontents
	\newpage
	\title{\Large \bf Forward Invariance of Sets for Hybrid Dynamical Systems \\ (Part II)}
	\author{Jun Chai and Ricardo G. Sanfelice}
	\maketitle

\tableofcontents

\newpage
\begin{abstract}
	This article presents tools for the design of control laws inducing robust controlled forward invariance of a set for hybrid dynamical systems modeled as hybrid inclusions.
	A set has the robust controlled forward invariance property via a control law if every solution to the closed-loop system that starts from the set stays within the set for all future time, regardless of the value of the disturbances.
	Building on the first part of this article, which focuses on analysis (Chai and Sanfelice, 2019), in this article, sufficient conditions for generic sets to enjoy such a property are proposed.
	To construct invariance inducing state-feedback laws, the notion of robust control Lyapunov function for forward invariance is defined.
	The proposed synthesis results rely on set-valued maps that include all admissible control inputs that keep closed-loop solutions within the set of interest.
	Results guaranteeing the existence of such state-feedback laws are also presented. Moreover, conditions for the design of continuous state-feedback laws with minimum point-wise norm are provided.
	Major results are illustrated throughout this article in a constrained bouncing ball system and a robotic manipulator application.
\end{abstract}

\section{Introduction}
\label{sec:intro}
\input{Sections/1-Intro.tex}

\section{Preliminaries}
\label{sec:HS}
\input{Sections/2-Huw.tex}

\section{Robust Controlled Forward Invariance for Hybrid Systems}
\label{sec:rcFI}
\input{Sections/3-RCFI.tex}

\subsection{CLF-based Approach for the Design of Robust Invariance-based Feedback Laws}
\label{sec:design}
\input{Sections/4-RCFI-Lya.tex}

\subsection{Existence of Pair $\kpair$ for Robust Controlled Forward Invariance}
\label{sec:rclf-FI}
\input{Sections/5-RCLF-FI.tex}
\subsection{Systematic Design of Pair $\kpair$ for Robust Controlled Forward Invariance}
\label{sec:rselect}
\input{Sections/6-Min-Norm.tex}
\section{Conclusion}
We propose methods for the design of controllers that render sets robust controlled forward invariant for hybrid dynamical systems.
The hybrid systems are modeled using differential and difference inclusions with state, control inputs, and disturbance constraints.
The robust controlled forward invariance properties are guaranteed by conditions on the data of the system, using CLFs for forward invariance.
The invariance property is guaranteed for the closed-loop system resulting from using a feedback controller.
When a set $K$ enjoys such properties, solutions to the closed-loop system evolve within the set they start from, even under the presence of disturbances.

Conditions on the data of the closed-loop system guaranteeing that sublevel sets of a given Lyapunov-like function are robustly forward invariant are presented.
Such conditions take advantage of the nonincreasing properties of $\Ly$ near the boundary of its sublevel sets.
To guarantee existence of nontrivial solution pairs from every point in such sublevel sets and completeness of every maximal solution pair, assumptions similar to those in \cite[Theorem 5.1]{PartI-Arxiv} are enforced.
When compared to the conditions in \cite[Theorem 5.1]{PartI-Arxiv}, on $\Cw$ required here are less restrictive as it does not require the flow set to be regular.

To systematically construct feedback pairs that render sets forward invariant uniformly in disturbances, we introduce control Lyapunov functions for forward invariance.
Such functions are not necessarily nonincreasing within the set to render forward invariant.
The proposed RCLF notions are conveniently used to derive conditions for the existence of continuous state-feedback laws inducing forward invariance.
The idea is to select feedback control from two carefully constructed set-valued maps, called the regulation maps.
Very importantly, the new RCLF notion is employed to synthesize state-feedback laws with pointwise minimum norm that effectively guarantee forward invariance.
For the stronger robust controlled forward invariance case, where completeness is required for every maximal solution pair within the set, a regulation map for flows involving the tangent cone of the flow set is derived from the well-known Nagumo Theorem.

Research on properties of the chosen selections using inverse optimality are undergoing.
Future research directions also include the development of barrier certificates for hybrid systems; see initial results in \cite{183}.

\balance
\bibliographystyle{unsrt}
\bibliography{bibs/CJref,bibs/RGSweb,bibs/CJintro}

\appendix
\label{sec:app}
\input{Sections/appendix.tex}

\end{document}

%% file: Packages.tex
\usepackage{graphics} 
\usepackage{graphicx} 
\usepackage{psfrag} 
\usepackage{epsfig} 
\usepackage{caption} 
\usepackage{subcaption} 
\usepackage{float} 
\usepackage{color} 
\usepackage[dvipsnames]{xcolor} 
\usepackage{paralist}

\usepackage{amsmath} 
\usepackage{latexsym} 
\usepackage{amssymb} 
\usepackage{amsxtra} 
\usepackage{abraces} 
\usepackage{mathtools}
\mathtoolsset{showonlyrefs=true}
\usepackage{subeqnarray}
\usepackage{wasysym}

\usepackage{enumitem} 
\usepackage[noadjust]{cite} 
\usepackage[breaklinks = true]{hyperref} 
\usepackage{ifthen}

\usepackage{balance} 
\usepackage{setspace} 

%% file: GenMacros.tex
\DeclareMathOperator*{\argmin}{\arg\min}

\newcommand{\ifeq}{\text{\normalfont if }}
\newcommand{\otherw}{\text{\normalfont otherwise}}
\newcommand{\ConfSp}[1]{\IfConf{\vspace{#1}}{ }}
\newcommand{\inter}{\text{\normalfont int}}

\newcommand{\dc}{\reals^n \times {\cal U}_c \times {\cal W}_c}
\newcommand{\dd}{\reals^n \times {\cal U}_d \times {\cal W}_d}
\newcommand{\Uc}{{\cal U}_c}
\newcommand{\Ud}{{\cal U}_d}

\renewcommand{\u}{_u}

\newcommand{\Cu}{C\u}

\newcommand{\Fu}{F\u}

\newcommand{\Wc}{{\cal W}_c}
\newcommand{\Wd}{{\cal W}_d}
\newcommand{\w}{_w}
\newcommand{\Hw}{\HS\w}
\newcommand{\Cw}{C\w}
\newcommand{\Dw}{D\w}
\newcommand{\Fw}{F\w}
\newcommand{\Gw}{G\w}
\newcommand{\Lw}{L\w}

\newcommand{\uw}{_{u,w}} 
\newcommand{\Huw}{\HS\uw}
\newcommand{\Cuw}{C\uw}
\newcommand{\Duw}{D\uw}
\newcommand{\Fuw}{F\uw}
\newcommand{\Guw}{G\uw}
\newcommand{\datauw}{(\Cuw, \Fuw, \Duw, \Guw)}

\newcommand{\Hcl}{\Hw}
\newcommand{\Ccl}{\Cw}
\newcommand{\Fcl}{\Fw}
\newcommand{\Dcl}{\Dw}
\newcommand{\Gcl}{\Gw}
\newcommand{\datacl}{(\Ccl, \Fcl, \Dcl, \Gcl)}

\newcommand{\wHS}{\widetilde{\HS}}
\newcommand{\wC}{\widetilde{C}}

\newcommand{\wD}{\widetilde{D}}

\newcommand{\Ly}{V}
\newcommand{\level}{r}
\newcommand{\rU}{\level^*}

\newcommand{\Ls}{L_\Ly(\level)}
\newcommand{\Lss}{L_\Ly(\rU)}
\newcommand{\Ir}{{\cal I} (\level, \rU)}

\newcommand{\K}{{\cal M}_\level}
\newcommand{\Mr}{{\cal M}_\level}

\newcommand{\Pwc}[1]{\Pi^w_c(#1)}
\newcommand{\Pwd}[1]{\Pi^w_d(#1)}
\newcommand{\Pc}[1]{\Pi_c(#1)}
\newcommand{\Pd}[1]{\Pi_d(#1)}
\newcommand{\wxc}{\Psi^w_c}

\newcommand{\wxuc}{\Phi^w_c}
\newcommand{\wxud}{\Phi^w_d}
\newcommand{\uxc}{\Psi^u_c}
\newcommand{\uxd}{\Psi^u_d}
\newcommand{\Lc}{\Theta_c}
\newcommand{\Ld}{\Theta_d}

\newcommand{\wLc}{\widetilde{\Theta}_c}

\newcommand{\wS}{\widetilde{S}}
\newcommand{\clfpair}{(\Ly, \rU)}

\newcommand{\kc}{\kappa_c}
\newcommand{\kd}{\kappa_d}
\newcommand{\rc}{\rho_c}
\newcommand{\rd}{\rho_d}
\newcommand{\kmc}{\kappa^m_c}
\newcommand{\kmd}{\kappa^m_d}
\newcommand{\Mc}{M_c}
\newcommand{\Md}{M_d}
\newcommand{\kpair}{(\kc, \kd)}

\newcommand{\hmax}{h_{\max}}
\newcommand{\hmin}{h_{\min}}
\newcommand{\vmax}{v_{\max}}
\newcommand{\Emax}{E_{\max}}

\newcommand{\ep}{e_p}
\newcommand{\umax}{u_{\max}}


%% file: Sections/1-Intro.tex
\subsection{Background and Motivation}
A set $K$ is forward invariant for a dynamical system if every solution to the system from $K$ stays in $K$.
Forward invariance properties have been key building blocks of stability theory since the early work by LaSalle and Krasovskii in 1960s.
In particular, scholars have studied forward invariance and controlled forward invariance together with stability in the sense of Lyapunov for different classes of dynamical systems.
In \cite{Blanchini.91.JOTA}, the author investigates the relationship between forward invariance and stability for uncertain constrained purely discrete-time and purely continuous-time systems.
In \cite{Bitsoris.Gravalau.95.Automatica}, inspired by stability analysis that uses a comparison principle, the authors derive conditions for the existence of forward invariant sets for constrained discrete-time nonlinear systems.
For a class of discrete-time systems, \cite{Limon.ea.02.CDC} establishes sufficient conditions for stability using invariant set theory, conditions that are applied to derive stability and feasibility of a model-predictive control problem with ``decaying perturbations.''
In \cite{Rodrigues.04.SCL}, stability of controlled invariant sets is achieved for piecewise-affine systems.

In recent years, several control applications have motivated control designs that go beyond Lyapunov stability and attractivity, in particular, that guarantee set invariance and safety properties under disturbances.
In \cite{li2016robust}, as a case study for manipulating genetic regulatory networks, robust invariance of a set is required to keep the states of a boolean network within a desired set.
For continuous-time monotone systems, \cite{Meyer.ea.16.Automatica} achieves energy efficiency in temperature control of ventilation in buildings via invariance analysis.
For nonlinear continuous-time systems, \cite{xu2015robustness} studies invariance applications in adaptive cruise control using control barrier functions.
Applications such as these have motivated our previous work in \cite{PartI-Arxiv}, where we develop systematic tools to verify forward invariance properties of sets without insisting on stability.
In addition, theoretical and computational results on robust controlled forward invariance are available in the literature for particular classes of systems.
Such a property guarantees that every solution to the closed-loop system stay within the set they started from, regardless of the values of the disturbances.
An extensive survey on control design for forward invariance is available in \cite{blanchini1999set}.
In \cite{hu2003composite}, the authors study invariance control for saturated linear continuous-time systems (the singular case is treated in \cite{Lin.Lv.07.TAC}).
Algorithms to estimate the maximal invariant set for discrete-time systems are given in \cite{Kerrigan.Maciejowski.00.CDC, Rakovic.ea.04.CDC, Collins.07.TCS}.
Methods for the design of invariance-based control laws for systems with inputs using control Lyapunov functions are less developed.
By solving convex optimization problems for linear discrete-time systems, \cite{Rakovic.ea.07.Automatica} and \cite{Rakovic.Baric.10.TAC} generate tools to verify and compute robust controlled invariant sets that are parametrized by a family of local control Lyapunov functions.

For systems exhibit switching dynamics, robust forward invariance analysis tools are applied to the design of feedback controllers in \cite{Lin.Antsaklis.02.CDC} for linear continuous-time systems that have a logic variable determining the mode of operation.
In \cite{stiver2001invariant}, methods to design invariance-inducing controllers exhibiting discrete events for continuous-time nonlinear systems are proposed.
The particular case of invariance-based control design for switched systems modeled as discrete-time systems (without perturbations) is treated in \cite{Benzaouia.ea.07.IJC}.
The authors in \cite{Julius.vanderSchaft.02.CDC} and \cite{Shang.04.ACC} propose algorithms to compute the controlled invariant sets for systems.

Invariance-based control for hybrid systems, which are systems that combine continuous and discrete dynamics, is much less explored, with only a few articles on the subject.
For reachability of desired sets, game theory techniques are applied in \cite{lygeros1999controllers} and \cite{Gao.ea.06.HSCC} to render sets controlled invariance for a class of hybrid systems with disturbances.
Similarly, barrier functions (and control barrier functions), which lead to controlled invariant sets, have been effectively employed in the study of safety for classes of hybrid systems \cite{wieland2007constructive}.
Moreover, in \cite{prajna2004safety} and \cite{kong2013exponential}, such functions are used for safety verification in hybrid automata with disturbances.

\ConfSp{-4mm}
\subsection{Contributions}
In \cite{PartI-Arxiv},
we formally define notions pertaining to robust forward invariance of sets for hybrid dynamical systems modeled as hybrid inclusions \cite{65}.
Sufficient conditions that apply to generic sets are presented therein.
In addition, we establish conditions to render the sublevel sets of Lyapunov-like functions forward invariant for hybrid systems without disturbances.
In this paper, continuing from \cite{PartI-Arxiv}, we focus on design of controllers that confer invariance properties presented therein, for hybrid systems given as in \cite{135}.
In particular, differential and difference inclusions with state, input, and disturbance constraints are used to model the continuous and discrete dynamics of hybrid systems, respectively.
More precisely, we consider hybrid systems with control inputs $u = (u_c, u_d) \in \Uc \times \Ud$ and  disturbances $w = (w_c,w_d) \in \Wc \times \Wd$ that are given by\footnote{The space for control inputs and disturbances are $\Uc \subset \reals^{m_c}, \Ud \subset \reals^{m_d}$ and $\Wc \subset \reals^{d_c}, \Wd \subset \reals^{d_d}$, respectively.} 
\begin{align}\label{eq:Huw}
	\Huw  \begin{cases}
		(x, u_c, w_c) \in \Cuw  & \dot{x} \ \ \in \Fuw(x, u_c, w_c)\\
		(x, u_d, w_d) \in \Duw  & x^+ \in \Guw(x, u_d, w_d),
	\end{cases}
\end{align}
where $x \in \reals^n$ is the state, $\Cuw \subset \reals^n \times \Uc \times \Wc$ and $\Duw  \subset \reals^n \times \Ud \times \Wd$ are called the flow and jump set, respectively, while $\Fuw$ and $\Guw$ are called the flow and jump map, respectively.
For this broad class of hybrid systems, the contributions made by this paper include:
\begin{enumerate}[leftmargin = \ItemSp]
	\item {\it Robust controlled forward invariance for $\Huw$ via $\kpair$}:
	we introduce the concept of robust controlled forward invariance.
	When a $\Huw$-admissible\footnote{A state-feedback pair $\kpair$, where $\kc : \reals^{n} \rightarrow \reals^{m_c}$ and $\kd : \reals^{n} \rightarrow \reals^{m_d}$, is said to be $\Huw$-admissible if the pair satisfies the dynamics of $\Huw$.} state-feedback pair $\kpair$ renders a set $K \subset \reals^n$
	robustly controlled forward invariant for the closed-loop system, the existence of a nontrivial solution pair
	from every possible initial condition is guaranteed.
	Moreover, every maximal solution pair (see Definition~\ref{def:solution}) that starts from the set is complete and stays within the set for all future (hybrid) time.
\NotConf{
	we provide sufficient conditions for verifying robust controlled forward invariance of a generic set $K \subset \reals^n$ for $\Huw$ via given feedback laws.
	Such a property holds for state component of each solution pair of the closed-loop hybrid system resulting from $\Huw$ being controlled by a $\Huw$-admissible state-feedback pair $\kpair$, which is given by
	\begin{align}\label{eq:Hcl}
		\Hcl \begin{cases}
			(x, w_c) \in \Ccl & \dot{x} \ \in \Fcl(x, w_c)\\
			(x, w_d) \in \Dcl & x^+ \! \in \Gcl(x, w_d),
		\end{cases}
	\end{align}
	where the set-valued maps $\Fcl(x,w_c) := \Fuw(x, \kc(x), w_c)$ and $\Gcl(x,w_d) := \Guw(x, \kd(x),$ $w_d)$ govern the continuous-time and discrete-time evolutions of the system on the sets
	$\Ccl  := \{(x, w_c) \in \reals^n \times \Wc: (x, \kc(x), w_c) \in \Cuw\}$
	and $\Dcl := \{(x, w_d) \in \reals^n \times \Wd: (x, \kd(x), w_d) \in \Duw\},$ respectively.
	Note that $\Hcl$ shares similar structure as the hybrid system $\Hw$ in (1); see \cite{PartI-Arxiv}.
	Applying results in \cite{PartI-Arxiv},
	we propose sufficient conditions guaranteeing that a feedback pair induces a set $K$ robust controlled forward invariance for $\Huw$.
	The challenges in deriving these results include:
	\begin{itemize}
		\item The possible set-valuedness of $\Fcl$ and $\Gcl$ and the nonunique solution pairs caused by existence of states and disturbances from where flows and jumps are both allowed (namely, the state component of $\Ccl$ and $\Dcl$ may overlap);
		\item The set $K$ ought to enjoy forward invariance properties over all possible disturbances for hybrid dynamical systems.
		More precisely, when flows occur, the pair $(x, w_c)$ belongs to $\Ccl$ and elements of $\Fcl(x, w_c)$ are required to point tangentially or inward the set $K$ regardless of the values for $w_c$.
		Similarly, when jumps occur, the pair $(x, w_d)$ belongs to $\Dcl$ and $\Gcl(x, w_d)$ has to map to points in $K$ regardless of the values of $w_d$.
	\end{itemize}
}
	\item {\it Robust forward invariance of sublevel sets of Lyapunov-like functions}:
	conditions to guarantee robust forward invariance properties that take advantage of the nonincreasing property of a Lyapunov-like function, $\Ly$, are proposed.
	As in \cite{PartI-Arxiv},
	we intersect the sublevel sets of the given function $\Ly$ with the state component of the flow and jump sets
	to define the set to be rendered robustly controlled forward invariant.
	Technical conditions are needed to guarantee the existence of nontrivial solution pairs from every point in such a set as well as to guarantee completeness of solution pairs.
	Note that these Lyapunov-like functions ought to satisfy inequalities over carefully constructed regions that allow for the potential increase in $\Ly$ in the interior of their sublevel sets.
	Moreover, compared to \cite[Theorem 5.1]{PartI-Arxiv},
	we further relax the regularity on the flow set via a constructive proof that employs properties of vectors in the tangent cone of the sets.
	\item {\it Existence of continuous state-feedback laws using robust control Lyapunov functions for forward invariance (RCLF for forward invariance)}:
	we present the concept of robust control Lyapunov function for forward invariance for the purpose of rendering a set robust controlled invariant.
	The proposed notion extends and is derived from the conditions in \cite{135} for asymptotic stability.
	Such a novel concept is exploited to determine sufficient conditions that lead to the existence of continuous state-feedback laws for robust controlled invariance.
	These conditions involve the data of the system and properly constructed set-valued maps in terms of $\Ly$--called the regulation maps.
	In particular, by assuring the existence of continuous selections from the said set-valued maps, forward invariance of sublevel sets of $\Ly$ is guaranteed.
	\item {\it Pointwise minimum norm selections as continuous state-feedback laws}:
	utilizing the regulation maps, we propose a pointwise minimum norm selection scheme to construct state-feedback laws that not only render the set robustly controlled forward invariant, but also are continuous.
\end{enumerate}

In summary, in this paper, we propose control synthesis methods for the purpose of rendering a set robustly controlled forward invariant for a general class of hybrid dynamical systems with disturbances.\footnote{The nominal version of the results in this paper appeared without proof in the conference article \cite{141} with a slightly different definition of the CLF for forward invariance.}
Major results are illustrated in two control design applications in which the dynamical systems can be modeled as hybrid inclusions as in \eqref{eq:Huw}.
More precisely, the results are illustrated in
	\begin{enumerate}
		\item {\it a constrained bouncing ball system,} for which the control goal is to maintain the ball to bounce back within a desired height range under the effect of an uncertain coefficient of restitution, and
		\item {\it a robotic manipulator interacting with an environment,} 
		for which the control goal is to guarantee that the end-effector only operates within a safe region.
	\end{enumerate}
For both applications, the designed state-feedback controllers induce robust forward invariance of sets describing the corresponding control objectives.
These applications are revisited multiple times to illustrate definitions, concepts and results.

Our results are also insightful for systems with purely continuous-time or discrete-time dynamics.
In fact, because of the generality of the hybrid inclusions framework, the results in this paper are applicable to broader classes of systems, such as those studied in \cite{fernandes1987remarks, blanchini1999set, bitsoris1988positive, hu2003composite}.

\ConfSp{-4mm}
\subsection{Organization and Notation}
The remainder of the paper is organized as follows.
Preliminaries about the considered class of hybrid systems is in Section~\ref{sec:HS}.
The robust controlled forward invariance notions and sufficient conditions to guarantee each notion are presented in Section~\ref{sec:rcFI}.
In Section~\ref{sec:rcFI-Lya}, sufficient conditions to induce robust forward invariance of sets are proposed for systems with a given Lyapunov-like function.
In Section~\ref{sec:rclf-FI}, the results on the existence of continuous state-feedback laws for robust controlled forward invariance are presented.
The pointwise minimum control law is in Section~\ref{sec:rselect}.

\medskip
\noindent{\bf Notation:}
Given a set-valued map $M : \reals^m \rightrightarrows \reals^n$, we denote the range of $M$ as $\rge M = \{ y \in \reals^n: y \in M(x), x \in \reals^m \}$, the domain of $M$ as $\dom M = \{ x \in \reals^m : M(x) \neq \emptyset \}$, and the graph of $M$ as $\gph M = \{ (x,y)\in \reals^m \times \reals^n : y \in M(x)\}$.
Given $\level \in \reals$, the $\level$-sublevel set of a function $\Ly: \reals^n \rightarrow \reals$ is $\Ls: = \{x \in \reals^n: \Ly(x) \leq \level \}$, $\Ly^{-1}(\level) = \{x \in \reals^n: \Ly(x) = \level \}$ denotes the $\level$-level set of $\Ly$, and, following the same notation in \cite[Section V]{PartI-Arxiv}, 
given a constant $\level \leq \rU$, we define the set $\Ir : = \{x \in \reals^n : \level \leq \Ly(x) \leq \rU \}$.
The closed unit ball around the origin in $\reals^n$ is denoted as $\ball$.
Given a closed set $K$, we denote the tangent cone of the set $K$ at a point $x \in K$ as $T_K(x)$.
The closure of the set $K$ is denoted as $\overline{K}$.
The set collecting all boundary points of a set $K$ is denoted by $\partial K$ and the set of interior points of $K$ is denoted by $\inter K$.
Given vectors $x$ and $y$, $(x, y)$ is equivalent to $[x^\top \ y^\top]^\top$.
Given a vector $x, |x|$ denotes the 2-norm of $x$.

%% file: Sections/2-Huw.tex
In this paper, we are interested in forward invariance properties of a set that are uniform in the disturbances $w$ for the closed-loop system $\Hcl$ in \eqref{eq:Hcl} resulting from controlling $\Huw$ in \eqref{eq:Huw} by a $\Huw$-admissible state-feedback pair $\kpair$.
Note that some properties and notions in this paper are clearly defined for the original (open-loop) hybrid system $\Huw$ with control inputs, while others are developed for the (perturbed) closed-loop system $\Hcl$.
In \eqref{eq:Huw}, sets $\Cuw$ and $\Duw$ define conditions that $x, u,$ and $w$ should satisfy for flows or jumps to occur, respectively.
The maps $\Fuw$ and $\Guw$ capture the system dynamics when in sets $\Cuw$ and $\Duw$, respectively.
For ease of exposition, for every $\star \in \{c,d\}$,
we define the projection of $S \subset \reals^n \times {\cal W}_\star$ onto $\reals^n$ as
$$\Pi^w_\star(S):= \{x \in \reals^n: (x, w_\star) \in S\},$$
and the projection of $S \subset \reals^n \times {\cal U}_\star \times {\cal W}_\star$ onto $\reals^n$ as
$$\Pi_\star(S):= \{x \in \reals^n: (x, u_c, w_c) \in S\}.$$
Given sets $\Cuw$ and $\Duw$, the set-valued maps $\wxuc: \reals^n \times {\cal U}_c \rightrightarrows \Wc$ and $\wxud: \reals^n \times {\cal U}_d \rightrightarrows \Wd$ are defined as
\begin{align}
	\begin{split}
		&\wxuc(x,u_c) := \{w_c \in \reals^{d_c} : (x, u_c, w_c) \in \Cuw\}, \\
		&\wxud(x,u_d) := \{w_d \in \reals^{d_d} : (x, u_d, w_d) \in \Duw\},
	\end{split}\label{eq:wxu}
\end{align}
for each $(x,u_c) \in \reals^n\times{\cal U}_c$ and each $(x,u_d) \in \reals^n\times{\cal U}_d$, respectively,
and the set-valued maps $\uxc: \reals^n \rightrightarrows \Uc$ and $\uxd: \reals^n \rightrightarrows \Ud$ are defined, for each $x \in \reals^n$, as
	$$\uxc(x) := \{u_c \in \reals^{m_c}: (x, u_c, w_c) \in \Cuw\},$$
	$$\uxd(x) := \{u_d \in \reals^{m_d}: (x, u_d, w_d) \in \Duw\},$$
respectively.

Solutions to a hybrid system $\Hcl$ as in \eqref{eq:Hcl} are parameterized by hybrid time domains $\cal E$, which are subsets of $\reals_{\geq 0} \times\nats$ that, for each $(T,J) \in {\cal E}, {\cal E} \cap ([0,T] \times \{0,1,...,J\})$ can be written as $\bigcup\limits^{J-1}_{j= 0} ([t_j,t_{j+1}],j)$ for some finite sequence of times $0 = t_0 \leq t_1 \leq t_2 ... \leq t_J.$
Moreover, following \cite[Definition 2.4]{65}, a hybrid arc $\phi$ is a function on a hybrid time domain that, for each $j \in \nats, t  \mapsto \phi(t, j)$ is absolutely continuous on the interval $I^j: = \{t : (t, j) \in \dom \phi\}$, where $\dom \phi$ denotes the hybrid time domain of $\phi$.

To make this paper self contained, we recall the solution pair concept in \cite[Definition 2.1]{PartI-Arxiv}.
\begin{definition}(solution pairs to $\Hcl$)\label{def:solution}
	A pair $(\phi, w)$ consisting of a hybrid arc $\phi$ and a hybrid disturbance $w = (w_c, w_d)$, with $\dom\phi = \dom w (= \dom (\phi, w)),$\footnote{Recall from \cite{PartI-Arxiv}, a hybrid disturbance $w$ is a function on a hybrid time domain that, for each $j \in \nats, t \mapsto w(t,j)$ is Lebesgue measurable and locally essentially bounded on the interval $\{t : (t,j) \in \dom w\}$.} is a solution pair to the hybrid system $\Hcl$ in \eqref{eq:Hcl} if $(\phi(0,0), w_c(0,0)) \in \overline{\Ccl}$ or $(\phi(0,0), w_d(0,0))\in \Dcl$, and
	\begin{enumerate}[label = (S\arabic*$\w$), leftmargin = \ItemLong]
		\item\label{item:Sw1} for all $j \in \nats$ such that $I^j$ has nonempty interior
		\begin{align*}
			&(\phi(t,j), w_c(t,j)) \in \Ccl &\mbox{for all } t \in \text{\normalfont int }I^j,\\
			&\frac{d\phi}{dt}(t,j) \in \Fcl(\phi(t,j), w_c(t,j)) &\mbox{for almost all } t \in I^j,
		\end{align*}
		\item\label{item:Sw2}  for all $(t,j) \in \dom\phi$ such that $(t, j+1) \in \dom\phi$,
		\begin{align*}
			&(\phi(t,j), w_d(t,j)) \in \Dcl\\
			&\phi(t,j+1) \in \Gcl(\phi(t,j), w_d(t,j)).
		\end{align*}
	\end{enumerate}
	In addition, a solution pair $(\phi, w)$ to $\Hcl$ is
	\begin{itemize}[leftmargin = \ItemSp]
		\item nontrivial if $\dom (\phi,w)$ contains at least two points;
		\item complete if $\dom (\phi,w)$ is unbounded;
		\item maximal if there does not exist another $(\phi, w)'$ such that $(\phi, w)$ is a truncation of $(\phi, w)'$ to some proper subset of $\dom(\phi, w)'$.
		\hfill $\square$
	\end{itemize}
\end{definition}
\IfConf{
	Given $K \subset \reals^n$,  $\sol_{\Hcl}(K)$ denotes the set that includes all maximal solution pairs $(\phi, w)$ to the hybrid system $\Hcl$ with $\phi(0,0) \in K$.
}{
	We use $\sol_{\Hcl}$ to represent the set of all maximal solution pairs to the hybrid system $\Hcl$ and, given $K \subset \reals^n$,  $\sol_{\Hcl}(K)$ denotes the set that includes all maximal solution pairs $(\phi, w)$ to the hybrid system $\Hcl$ with $\phi(0,0) \in K$.
	
Next, we list \cite[Proposition 3.4]{PartI-Arxiv} as below, which presents conditions guaranteeing existence of nontrivial solution pairs to $\Hcl$ from every initial state $\xi \in \overline{\Pwc{\Ccl}}\cup \Pwd{\Dcl}$.
This result is used in later sections to characterize all possibilities for maximal solution pairs to $\Hcl$.
\begin{proposition}(basic existence under disturbances)\label{prop:wexistence}
	Consider a hybrid system $\Hcl = \datacl$ as in \eqref{eq:Hcl}.
	Let $\xi \in \overline{\Pwc{\Ccl}} \cup \Pwd{\Dcl}$.
	If $\xi \in \Pwd{\Dcl}$, or
	\begin{enumerate}[label = (VC$\w$), leftmargin=3\parindent]
		\item\label{item:VCw} there exist $ \varepsilon > 0$, an absolutely continuous function $\widetilde{z}: [0,\varepsilon] \rightarrow \reals^n$ with $\widetilde{z}(0) = \xi$, and a Lebesgue measurable and locally essentially bounded function $\widetilde{w}_c : [0,\varepsilon] \rightarrow \Wc$ such that $(\widetilde{z}(t), \widetilde{w}_c(t))\in \Ccl$ for all $t \in (0,\varepsilon)$ and $\dot{\widetilde{z}}(t) \in \Fcl(\widetilde{z}(t), \widetilde{w}_c(t))$ for almost all $t \in [0,\varepsilon]$, where $\widetilde{w}_c(t) \in \wxc(\widetilde{z}(t))$ for every $t \in [0,\varepsilon]$,
	\end{enumerate}
	then, there exists a nontrivial solution pair $(\phi, w)$ from the initial state $\phi(0,0) = \xi$.
	If $\xi \in \Pwd{\Dcl}$ and \ref{item:VCw} holds for every $\xi \in \overline{\Pwc{\Ccl}} \setminus \Pwd{\Dcl}$, then there exists a nontrivial solution pair to $\Hcl$ from every initial state $\xi \in \overline{\Pwc{\Ccl}} \cup \Pwd{\Dcl}$, and every solution pair $(\phi, w) \in \sol_{\Hcl}$ from such points satisfies exactly one of the following:
	\begin{enumerate}[label = \alph*), leftmargin = \ItemSp]
		\item\label{item:a} the solution pair $(\phi,w)$ is complete;
		\item\label{item:b} $(\phi,w)$ is not complete and ``ends with flow'': with $(T,J) = \sup\dom(\phi, w)$, the interval $I^J$ has nonempty interior, and either
		\begin{enumerate}[label = b.\arabic*)]
			\item\label{item:b.1} $I^J$ is closed, in which case either
			\begin{enumerate}[label = b.1.\arabic*)]
				\item\label{item:b.1.1} $\phi(T,J) \in \overline{\Pwc{\Ccl}}\setminus(\Pwc{\Ccl} \cup \Pwd{\Dcl})$, or
				\item\label{item:b.1.2} from $\phi(T,J)$ flow within $\Pwc{\Ccl}$ is not possible, meaning that there is no $\varepsilon > 0$, absolutely continuous function $\widetilde{z}: [0,\varepsilon] \rightarrow \reals^n$ and a Lebesgue measurable and locally essentially bounded function $\widetilde{w}_c : [0,\varepsilon] \rightarrow \Wc$ such that $\widetilde{z}(0) = \phi(T,J)$, $(\widetilde{z}(t), \widetilde{w}_c(t))\in \Ccl$ for all $t \in (0,\varepsilon)$, and $\dot{\widetilde{z}}(t) \in \Fcl(z(t), \widetilde{w}_c(t))$ for almost all $t \in [0,\varepsilon]$, where $\widetilde{w}_c(t) \in \wxc(\widetilde{z}(t))$ for every $t \in [0,\varepsilon]$, or
			\end{enumerate}
			\item\label{item:b.2} $I^J$ is open to the right, in which case $(T,J) \notin \dom(\phi, w)$ due to the lack of existence of an absolutely continuous function $\widetilde{z} : \overline{I^J} \rightarrow \reals^n$ and a Lebesgue measurable and locally essentially bounded function $\widetilde{w}_c : [0,\varepsilon] \rightarrow \Wc$ satisfying $(\widetilde{z}(t), \widetilde{w}_c(t)) \in \Ccl$ for all $t \in \inter I^J$, $\dot{\widetilde{z}}(t) \in  \Fcl(\widetilde{z}(t), \widetilde{w}_c(t))$ for almost all $t \in I^J,$ and such that $\widetilde{z}(t) = \phi(t,J)$ for all $t \in I^J $, where $\widetilde{w}_c(t) \in \wxc(\widetilde{z}(t))$ for every $t \in [0,\varepsilon]$;
		\end{enumerate}
		\item\label{item:c} $(\phi, w)$ is not complete and ``ends with jump'': with $(T,J) = \sup\dom(\phi, w) \in \dom(\phi, w)$, $(T,J-1) \in \dom (\phi, w)$, and either\footnote{As a consequence of $(\phi, w)$ ending with a jump, which implies that $\phi(T,J) \notin \Pwd{\Dcl}$, $\phi(T,J) \in \overline{\Pwc{\Ccl}}\setminus \Pwd{\Dcl}$ is under the condition in case \ref{item:c.2}.}
		\begin{enumerate}[label = c.\arabic*)] 
			\item\label{item:c.1} $\phi(T,J) \notin \overline{\Pwc{\Ccl}} \cup \Pwd{\Dcl}$, or
			\item\label{item:c.2} $\phi(T,J) \in \overline{\Pwc{\Ccl}}\setminus \Pwd{\Dcl}$, and from $\phi(T,J)$ flow within $\Pwc{\Ccl}$ as defined in \ref{item:b.1.2} is not possible.
		\end{enumerate}
	\end{enumerate}
\end{proposition}
}

The following regularity conditions on the system data of a hybrid system $\Hcl$ as in \eqref{eq:Hcl} are considered in some forthcoming results.
These conditions guarantee robustness of asymptotic stability of compact sets with respect to small perturbations; see \cite[Chapter 6]{65} for details.
\begin{definition}(hybrid basic conditions)\label{def:hbc}
	A hybrid system $\Hcl= \datacl$ is said to satisfy the hybrid basic conditions if its data satisfies
	\begin{enumerate}[label = (A\arabic*$_w$), leftmargin=\ItemLong]
		\item\label{item:A1} $\Ccl$ and $\Dcl$ are closed subsets of $\reals^n \times\Wc$ and $\reals^n \times \Wd$ respectively;
		\item\label{item:A2} $\Fcl: \reals^n \times \reals^{d_c} \rightrightarrows \reals^n$ is outer semicontinuous\footnote{See Definition~\ref{def:osc} in Appendix.} relative to $\Ccl$ and locally bounded, and for all $(x,w_c) \in \Ccl, \Fcl(x,w_c)$ is nonempty and convex;
		\item\label{item:A3} $\Gcl: \reals^n \times \reals^{d_d} \rightrightarrows \reals^n$ is outer semicontinuous relative to $\Dcl$ and locally bounded, and for all $(x,w_d) \in \Dcl, \Gcl(x,w_d)$ is nonempty. \hfill $\square$
	\end{enumerate}
\end{definition}

To obtain properties \ref{item:A1}-\ref{item:A3} in Definition~\ref{def:hbc} for $\Hcl$, we have the following immediate result.
\begin{lemma}(hybrid basic conditions)\label{lem:hbc}
	Suppose $\kc:\Pc{\Cuw} \rightarrow \Uc$ and $\kd:\Pd{\Duw} \rightarrow \Ud$ are continuous and $\Huw = \datauw$ is such that
	\begin{enumerate}[label = (A\arabic*'), leftmargin=\ItemLong]
		\item\label{item:A1p} $\Cuw$ and $\Duw$ are closed subsets of $\dc$ and $\dd$, respectively;
		\item\label{item:A2p} $\Fuw : \reals^n \times \reals^{m_c} \times \reals^{d_c} \rightrightarrows \reals^n$ is outer semicontinuous relative to $\Cuw$ and locally bounded, and for every $(x, u_c, w_c) \in \Cuw, \Fuw (x, u_c, w_c)$ is nonempty and convex;
		\item\label{item:A3p} $\Guw : \reals^n \times \reals^{m_d} \times \reals^{d_d} \rightrightarrows \reals^n$ is outer semicontinuous relative to $\Duw$ and locally bounded, and for every $(x, u_d, w_d) \in \Duw, \Guw (x, u_d, w_d)$ is nonempty.
	\end{enumerate}
	Then, $\Hcl$ satisfies conditions \ref{item:A1}-\ref{item:A3} in Definition~\ref{def:hbc}.
\end{lemma}

%% file: Sections/3-RCFI.tex
In this section, we first provide conditions guaranteeing that a static state-feedback pair renders robustly forward invariant (in the appropriate sense) a set for the closed-loop system.
These conditions involve the $\Huw$-admissible state-feedback pair $\kpair$, the data of the closed-loop system it leads to, which is denoted $\Hcl$, and the set $K$ to render robustly forward invariant.
We also provide conditions guaranteeing the existence of such feedbacks as well as a method for their systematic design.

Provided with a $\Huw$-admissible state-feedback pair $\kpair$ the closed-loop hybrid system resulting from $\Huw$ in \eqref{eq:Huw} is given by
	\begin{align}\label{eq:Hcl}
	\Hcl \begin{cases}
	(x, w_c) \in \Ccl & \dot{x} \ \in \Fcl(x, w_c)\\
	(x, w_d) \in \Dcl & x^+ \! \in \Gcl(x, w_d),
	\end{cases}
	\end{align}
	where the set-valued maps $\Fcl(x,w_c) := \Fuw(x, \kc(x), w_c)$ and $\Gcl(x,w_d) := \Guw(x, \kd(x), w_d)$ govern the continuous and discrete dynamics of the system on the sets
	$\Ccl  := \{(x, w_c) \in \reals^n \times \Wc: (x, \kc(x), w_c) \in \Cuw\}$
	and $\Dcl := \{(x, w_d) \in \reals^n \times \Wd: (x, \kd(x), w_d) \in \Duw\},$ respectively.
	Note that $\Hcl$ shares similar structure as the hybrid system $\Hw$ in (1) of \cite{PartI-Arxiv}.
	To this end, to make the paper self contained, we recall the following notions from \cite[Definition 3.2]{PartI-Arxiv} which are used in this section.
\begin{definition}(robust forward (pre-)invariance of $\Hcl$)\label{def:rFI}
	The set $K \subset \reals^n$ is said to be {\em robustly forward pre-invariant for $\Hcl$} if every $(\phi, w)\in \sol_{\Hcl}(K)$ is such that $\rge \phi \subset K$.
	The set $K \subset \reals^n$ is said to be {\em robustly forward invariant for $\Hw$} if for every $x \in K$ there exists a solution pair to $\Hcl$ and every $(\phi, w)\in \sol_{\Hcl}(K)$ is complete and such that $\rge \phi \subset K$.
	\hfill $\square$
\end{definition}

Building from this definition, we introduce the following robust controlled forward invariance notions.

\begin{definition}(robust controlled forward (pre-)invariance of $\Huw$) \label{def:rcFI}
	The set $K \subset \reals^n$ is said to be robustly controlled forward pre-invariant for $\Huw$ as in \eqref{eq:Huw} via a state-feedback pair $\kpair$ if the set $K$ is robustly forward pre-invariant for the resulting closed-loop system $\Hcl$.
	The set $K \subset \reals^n$ is said to be robustly controlled forward invariant for $\Huw$ via a state-feedback pair $\kpair$ as in \eqref{eq:Huw} if the set $K$ is robustly forward invariant for the resulting closed-loop system $\Hcl$.
	\hfill $\square$
\end{definition}

\begin{remark}
	As mentioned in Section~\ref{sec:intro}, our notions apply to a more general class of systems, in particular, continuous-time, discrete-time, and hybrid systems with set-valued dynamics.
	Very importantly, compared to \cite[Definition 2.3]{blanchini1999set}, \cite[Definition 8]{Meyer.ea.16.Automatica} (for continuous-time systems) or \cite[Definition 1]{rakovic2004computation} (for discrete-time systems), our notions do not require uniqueness of solutions to the closed-loop system.
	\NotConf{
	In addition, notions of weak robust controlled forward (pre-)invariance for $\Huw$ can also be derived from the weak robust forward (pre-) invariance of $\Hw$ in \cite[Definition 3.1]{PartI-Arxiv}, following Definition~\ref{def:rcFI}.
	Note that all of the results in this paper naturally apply to hybrid systems without disturbances.
}
\end{remark}

Throughout this paper, we demonstrate our main results in two control design problems for mechanical systems, namely, a constrained bouncing ball moving vertically that is controlled by impacts at zero height; and a robotic manipulator interacting with a surface.

\begin{example}(Constrained bouncing ball system)\label{ex:bball-model}
	Consider the bouncing ball system shown in Figure~\ref{fig:bball}.
	We attach one end of a nonelastic string with length $\hmax$ to zero height and the other end to a ball.
	The ball can only travel vertically and is controlled by impacts at zero height.
\ConfSp{-3mm}
	\IfConf{
		\setlength{\unitlength}{0.15\textwidth}
	}{
		\setlength{\unitlength}{0.3\textwidth}}
	\begin{figure}[H]
		\centering
		\scriptsize
		\includegraphics[width= \unitlength]{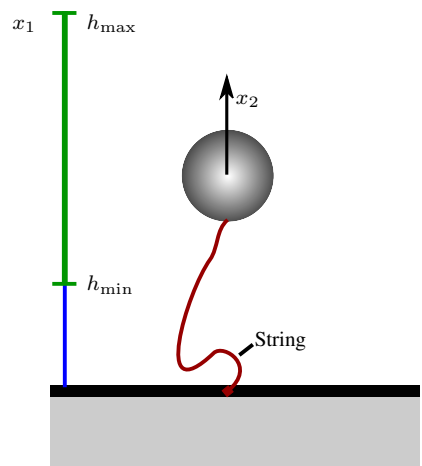}
		\put(-1.1,1.2){$x_1$}
		\put(-0.9,1.2){$\hmax$}
		\put(-0.5,1){$x_2$}
		\put(-0.45,0.35){String}
		\put(-0.9, 0.5){$\hmin$}
		\caption{Bouncing ball system configuration.}
		\label{fig:bball}
	\end{figure}
	\ConfSp{-5mm}
	Compared to a typical bouncing ball system \cite[Example 1.1]{65}, the model considered here has an additional ``pulling phase'' when the ball reaches the height $\hmax$ with possibly nonzero velocity.
	The possible pulls from the string at height $\hmax$ and the impacts between the ball and the controlled surface both lead to jumps of the state.
	In addition to assuming unitary mass of the ball and negligible weight of the string, forces, and friction, we consider the following:
	\begin{enumerate}[label = C\arabic*)]
		\item At impacts with the ground, the uncertain coefficient of restitution is within the range $[e_1, e_2]$, where $0 < e_1 < e_2 < 1$;
		\item\label{item:string} The string breaks when the ball pulls with velocity larger than $\vmax$;
		\item At pulls of the string, the restitution coefficient is $\ep \in (0, 1]$.
	\end{enumerate}
With $x = (x_1, x_2) \in \reals^2$, $x_1$ and $x_2$ model the height and velocity of the ball, respectively.
Then, with gravity constant $\gamma > 0$, the flow map is defined on $\reals_{\geq 0} \times \reals$ and is given by\footnote{Note that since there are no disturbances and inputs for flow, we omit the subscripts for $F$ and $C$ in this model.}
$$F(x) := (x_2, -\gamma).$$
To formulate the flow and jump set, we define a function $E: \reals^2 \rightarrow \reals$ that describes the total energy of the system as \IfConf{
	$E(x) = 0.5x_2^2 + \gamma x_1, \quad \forall x \in \reals^2.$
}{
follows:
\begin{align}\label{eq:BpE}
E(x) = \frac{x_2^2}{2}  + \gamma x_1 &\quad \forall x \in \reals^2.
\end{align}
}
According to \ref{item:string}, the string remains attached to the ball when $x_1 \in [0, \hmax]$ and $x_2 \leq \vmax$, i.e., $E(x) \leq \Emax$ with $\Emax := E(\hmax, \vmax)$.
After impacts with the controlled surface, the height of the ball $x_1$ remains unchanged, while the velocity $x_2$ is updated based on a function of the uncertain coefficient of restitution, which is treated as a disturbance $w_d \in \Wd :=[e_1, e_2]$, and the control input $u_d \in \Ud :=[0, \umax]$ with $\umax = \sqrt{2 \Emax}$, which represents the velocity change caused by the controlled surface.
Hence, we model impacts between the ball and the controlled surface as
\begin{align}\label{eq:BpJumpM1}
G_1(x,u_d, w_d) : = (x_1, u_d - w_dx_2)
\end{align}
when $x_1 = 0$ and $x_2  \leq 0$.
Before every impact, $x_2$ is nonpositive, and, after each impact, it is updated according to $G_1$.
Then, with a small constant $0 < \delta_p < \vmax$, the map
\begin{align}\label{eq:BpJumpM2}
G_2(x) : = (x_1, \min\{- \ep x_2, -\delta_p\})
\end{align}
models the pulls between the ball and the string when $x_1 = \hmax$ and $x_2 \in [0, \vmax]$.
Since before every pull, $x_2$ is nonnegative, after each pull the ball velocity reverses its sign and is updated according to $G_2$.
Note that since closed jump sets are preferred as suggested in \ref{item:A1} of Definition~\ref{def:hbc}, we only allow the $x_2$ component to jump to a strictly negative value that is lower bounded (and controllable) by $-\delta_p < 0.$

Then, the hybrid system $\Huw = (C,F,\Duw,\Guw)$ has $x = (x_1, x_2)$ as the state, $u_d$ as the control input and $w_d$ as the disturbance with $(x, u_d, w_d) \in {\cal X} = \reals^2 \times\Ud \times \Wd$ and dynamics given by
\begin{align}\label{eq:Bplant}
		\dot{x} &= F(x) 					 &x&\in C,\\
		x^+ &= \Guw(x,u_d, w_d) 	&(x, u_d, w_d) &\in \Duw,
\end{align}
where the flow set $C$ is given by
\begin{align}\label{eq:BC}
C: = \{x \in \reals^2 : 0 \leq x_1 \leq \hmax, E(x) \leq \Emax\},
\end{align}
the jump set $\Duw$ is given by $\Duw: = D^1_{u,w} \cup D^2_{u,w}$ with
\IfConf{
$D^1_{u,w} :=\{(x,u_d, w_d) \in {\cal X} : x_1 = 0, x_2 \in [-\sqrt{2\Emax}, 0]\},$
$D^2_{u,w} :=\{(x,u_d, w_d) \in {\cal X} : x_1 = \hmax, x_2 \in [0, \vmax]\},$
}{
\begin{align}\label{eq:BD}
&D^1_{u,w} :=\{(x,u_d, w_d) \in {\cal X} : x_1 = 0, x_2 \in [-\sqrt{2\Emax}, 0]\},\\
&D^2_{u,w} :=\{(x,u_d, w_d) \in {\cal X} : x_1 = \hmax, x_2 \in [0, \vmax]\},
\end{align}
}
and the jump map $\Guw$ is given by
\begin{align}\label{eq:BG}
\Guw(x,u_d, w_d) : =
\begin{cases}
G_1(x,u_d, w_d) &  \ifeq (x,u_d, w_d) \in D^1_{u,w}\\
G_2(x) &  \ifeq (x,u_d, w_d) \in D^2_{u,w}.
\end{cases}
\end{align}

\vspace{-4mm}
\noindent
We have the following control design goal: under the presence of disturbances $w_d$, design a feedback law assigning $u_d$ such that when the ball has initial condition $x(0,0) = (x_1(0,0), x_2(0,0))$ with $x_1(0,0) \in [\hmin, \hmax]$ and $E(x(0,0)) \in [0,\Emax]$, the string remains attached to the ball, and the peak height of the ball after each bounce is at least $\hmin$.
\hfill $\triangle$
\end{example}

	The next example presents an control design application with a control input that, unlike the system in Example~\ref{ex:bball-model}, is only active during flows.
\begin{example}(Robotic manipulator interacting with the environment)\label{ex:hmass-model}
	Consider a robotic manipulator interacting with a static working environment.
	As described in \cite[Section II.A]{13}, the interaction between the robotic manipulator and the working environment is captured by
	\begin{align*}
		\widetilde{M}(\theta) \ddot{x}_1 + \widetilde{C}(\theta, \dot{\theta}) \dot{x}_1 + \widetilde{N}(\theta, \dot{\theta}) = f_a - f_c,
	\end{align*}
	where $\widetilde{M}, \widetilde{C},$ and $\widetilde{N}$ represent the inertia matrix, the Coriolis matrix, and external forces (including the gravity) acting on the robotic arm joints, respectively.
	The term $f_a$ represents the actuator force and $f_c$ is the contact force.
	The state variable $x_1$ is the position of the end-effector of the manipulator and $\theta$ is the angle displacement of the joint.
	\ConfSp{-2mm}
	\IfConf{
	\setlength{\unitlength}{0.44\textwidth}
}{
	\setlength{\unitlength}{0.6\textwidth}}
	\begin{figure}[H]
		\centering
		\scriptsize
		\includegraphics[width= \unitlength]{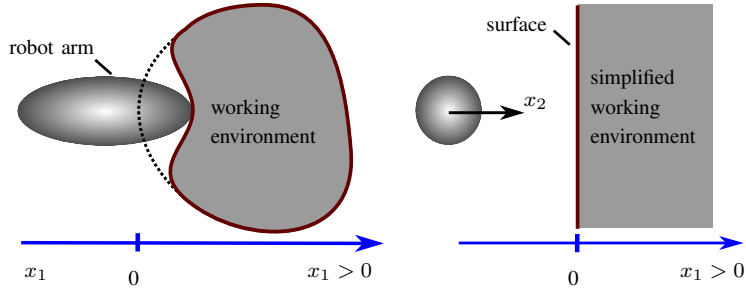}
		\put(-0.36,0.31){surface}
		\put(-0.22,0.24){simplified}
		\put(-0.22,0.2){working}
		\put(-0.22,0.16){environment}
		\put(-0.73,0.2){working}
		\put(-0.73,0.16){environment}
		\put(-1,0.28){robot arm}
		\put(-0.98, -0.02){$x_1$}
		\put(-0.6, -0.02){$x_1 > 0$}
		\put(-0.1, -0.02){$x_1 > 0$}
		\put(-0.31,0.21){$x_2$}
		\put(-0.25,-0.03){0}
		\put(-0.84,-0.03){0}
		\caption{Robotic manipulator system.}
		\label{fig:ptmass}
	\end{figure}
\ConfSp{-2mm}
	To stabilize some of the internal and external forces of the manipulator, a commonly used inner feedback law of the form
	\begin{align*}
		f_a = u_c + \widetilde{C}(\theta, \dot{\theta}) \dot{x}_1 + \widetilde{N}(\theta, \dot{\theta})
	\end{align*}
	is applied, see e.g. \cite{cavusoglu1997hybrid, tarn1996force}, which leads to
	\begin{align}\label{eq:raPlant}
		\widetilde{M}(\theta) \ddot{x}_1 = u_c - f_c.
	\end{align}
	Hence, the system dynamics are simplified to the interaction between the manipulator's end-effector and the working environment.
	Without loss of generality, only the constrained motion along a straight line is considered.
	More precisely, as depicted in Figure~\ref{fig:ptmass}, the simplified system consists of a point mass with unitary mass that only moves horizontally, and a elastic surface that represents the working environment.
	
	To mimic the different effects of elastic and plastic deformations of the working environment, a velocity threshold $\overline{v} > 0$ is introduced.
	More precisely, when the reaction stress of the material caused by the contact exceeds $\overline{v}$,  an impact occurs \cite{zhang2002modeling}.
	Similar to Example~\ref{ex:bball-model}, the impact is modeled using an uncertain coefficient of restitution within the range ${\cal W}_d := [e_1, e_2]$, where $0 < e_1< e_2< 1$.
	
	When the velocity is smaller than $\overline{v}$, the manipulator pushes against the surface, which results in a nonzero contact force $f_c$.
	With the (positive) elastic and viscous parameters of the contact denoted by $k_c$ and $b_c$, respectively, the discontinuous contact force is given by
	\begin{align}
		f_c(x) =
		\begin{cases}
			k_cx_1 + b_cx_2 & \ifeq x_1 \geq 0\\
			0 & \ifeq x_1 < 0.
		\end{cases}
	\end{align}
	For the resulting hybrid model to satisfy the hybrid basic conditions in Lemma~\ref{lem:hbc}, we consider the Filippov regularization of the contact force $f_c$ (see \cite[Chapter 4]{65}), which is given by
	\begin{align}
		f^r_c(x) =
			\begin{cases}
			k_cx_1 + b_cx_2 & \ifeq x_1 > 0\\
			\overline{\con}\{0, b_cx_2\} & \ifeq x_1 = 0\\
			0 & \ifeq x_1 < 0.
			\end{cases}
	\end{align} 
	Combining the above constructions, we model the dynamics of the manipulator as a hybrid system with input affecting the flows only and disturbances affecting the the jump only, i.e., $\Huw = (\Cu, \Fu, \Dw, \Gw)$.
	To this end, let the state variable be $x = (x_1, x_2) \in \reals^2$, where $x_1$ and $x_2$ represent the horizontal position and velocity of the point mass, respectively: see Figure~\ref{fig:ptmass}.
	The input force $u_c$ applied to the point mass is bounded and constrained to the set ${\cal U}_c := [-f_{max}, f_{max}]$.
	Using \eqref{eq:raPlant} and assuming that the inertia matrix is the identity, the flow map is given by $F_u(x, u_c) := (x_2, u_c - f^r_c(x))$.
	The flow set is given as\footnote{Note that noise in the applied input force at the point mass can be modeled as a disturbance $w_c$, however, we omit it for simplicity.}
	\begin{align}\label{eq:raFlowS}
		\Cu := &\{(x, u_c)\in \reals^2 \times {\cal U}_c: x_1 \leq 0\} \bigcup \\
				&\{(x, u_c)\in \reals^2 \times {\cal U}_c: x_1 \geq 0, x_2 \leq \overline{v}\}.
	\end{align}
	The jump set describes the condition that leads to an impact as discussed earlier, and it is given by
	\begin{align}\label{eq:raJumpS}
		\Dw := \{ (x, w_d) \in \reals^2 \times {\cal W}_d : x_1 \geq 0, x_2 \geq \overline{v} \}.
	\end{align}
	At such points, a jump happens according to the jump map
	\IfConf{
	$\Gw(x, w_d) := (x_1, -w_dx_2).$
}{
	\begin{align}\label{eq:raJumpM}
		\Gw(x, w_d) := (x_1, -w_dx_2).
	\end{align}
}
	Our goal is to design $u_c$ such that, regardless of whether the manipulator is in contact with the work environment or not, the end-effector stays within a safe region.
	\hfill $\triangle$
\end{example}

\NotConf{
\subsection{Conditions on a Pair $\kpair$ for Robust Controlled Forward Invariance}
Our first result consists of applying \cite[Theorem~4.15 and Lemma~4.12]{PartI-Arxiv} to derive conditions that a pair $\kpair$, along with the data of the hybrid system $\Huw$ and a given set $K$, should satisfy for $K$ to be robustly controlled invariant.
Though the result is not necessarily a systematic design tool, it provides checkable solution-independent conditions.

\begin{corollary}(robust controlled forward (pre-)invariance)\label{coro:rcFI}
	Consider a hybrid system $\Huw = \datauw$ as in \eqref{eq:Huw} and a $\Huw$-admissible state-feedback pair $\kpair$.
	Let the closed-loop system $\Hcl = \datacl$ satisfy the conditions in Definition~\ref{def:hbc}.
	Furthermore, suppose $K \subset \reals^n$ is a closed subset of $\overline{\Pwc{\Ccl}}\cup \Pwd{\Dcl}$ and $\Fcl$ is locally Lipschitz\footnote{See \cite[Definition A.4]{PartI-Arxiv}.} on $((\partial K + \delta \ball)\times \Wc) \cap \Ccl$ for some $\delta > 0$.
	Then, the set $K$ is robustly controlled forward pre-invariant for $\Huw$ via $\kpair$ if $K$ and $\datacl$ are such that
	\begin{enumerate}[label = \ref{coro:rcFI}.\arabic*), leftmargin=2.8\parindent]
		\item\label{item:rcFI1} For every $\xi \in (\partial K) \cap \Pwc{\Ccl}$, there exists a neighborhood $U$ of $\xi$ such that $\wxc(x) \subset \wxc(\xi)$ for every $x \in U \cap \Pwc{\Ccl}$;
		\item\label{item:rcFI2} For every $(x, w_d) \in (K \times \Wd) \cap \Dcl, \Gcl(x, w_d) \subset K$;
		\item\label{item:rcFI3} For every $(x, w_c) \in ((\partial(K \cap \Pwc{\Ccl}) \times \Wc) \cap \Ccl)\setminus \Lw, \Fcl(x,w_c) \subset T_{K \cap \Pwc{\Ccl}}(x)$, where $\Lw := \{(x,w_c) \in \Ccl : x \in \partial \Pwc{\Ccl}, \Fcl(x, w_c) \cap T_{\Pwc{\Ccl}}(x) = \emptyset\}$.
	\end{enumerate}
		Moreover, $K$ is robustly controlled forward invariant for $\Huw$ via $\kpair$ if, in addition
		\begin{enumerate}[label = \ref{coro:rcFI}.\arabic*), resume, leftmargin=2.8\parindent]
			\item\label{item:rcFI4} $(K \times \Wc) \cap \Ccl$ is compact, or
			$\Fcl$ has linear growth on $(K \times \Wc) \cap \Ccl$; and
			\item\label{item:rcFI5} $K \cap \Pwc{\Lw} \subset \Pwd{\Dcl}$.
		\end{enumerate}
\end{corollary}

\begin{proof}
	The proof exploits results in \cite{PartI-Arxiv}.
	Namely, applying \cite[Theorem~4.15]{PartI-Arxiv}, we show that the assumptions and conditions~\ref{item:rcFI1}-\ref{item:rcFI3} in Corollary~\ref{coro:rcFI} together imply the set $K$ is robustly pre-forward invariant for the closed-loop system $\Hcl$.
	In particular, $K \cap \Pwc{\Ccl}$ is closed since $K$ and $\Ccl$ are closed sets.
	Because of item~\ref{item:A2} and the assumption that $0 \in \wxc(x)$ for every $x \in \Pc{\Cuw}$, \cite[Assumption 4.10]{PartI-Arxiv} holds for $\Ccl, \Dcl, \Fcl$ and $K$.
	Note that in proof of \cite[Theorem 4.15]{PartI-Arxiv}, the locally Lipschitz property of $\Fcl$ in $w_c$ is only used on set $((\partial K + \delta \ball)\times \Wc) \cap \Ccl$ rather than on $\Ccl$.
	Hence, applying \cite[Theorem 4.15]{PartI-Arxiv}, since \ref{item:rcFI2} and \ref{item:rcFI3} imply 4.15.1) and 4.15.2), respectively, set $K$ is robustly controlled forward pre-invariant for $\Huw$ via $\kpair$ by Definition~\ref{def:rcFI}.
	
	With the addition of item \ref{item:rcFI4}, \cite[Lemma~4.12]{PartI-Arxiv} implies solution pairs are bounded in finite time.
	Then, item \ref{item:rcFI5} guarantees existence of nontrivial solution pairs from every $x \in K$ by guaranteeing jump is possible from every $x \in (K \cap \Pwc{\Lw})$.
	Therefore, $K$ is robustly forward invariant for $\Hcl$ and robustly controlled forward invariant for $\Huw$ via  $\kpair$.
\end{proof}

\begin{remark}
	The locally Lipschitzness of the set-valued map $\Fcl$ is crucial to make sure that every solution pair stays in the set $K$ during flows as shown in proof of \cite[Theorem 4.15]{PartI-Arxiv}.
	In addition, we refer readers to the example provided below \cite[Theorem 3.1]{blanchini1999set}, which shows that, even though $f(x) \in T_K$, a continuous-time system has solutions that leave a set due to the absence of locally Lipschitzness of the right-hand side of a continuous-time system.
	In addition, condition \ref{item:rcFI1} guarantees such property uniformly in $w_c$ (see the proof of \cite[Theorem 4.15]{PartI-Arxiv}.
\end{remark}

We use the next example to illustrate Corollary~\ref{coro:rcFI}.
\begin{example}(nonlinear planar system with jumps)\label{ex:rcFI}
	Consider a hybrid system $\Huw$ with flow map
	$$\Fuw(x, u_c, w_c) :=
	\left\{ \begin{bmatrix} x_1^2 - \gamma \\ x_1x_2 \end{bmatrix}
	u_c w_c : \gamma \in [3,4]\right\}$$
	defined for every $(x, u_c, w_c) \in \Cuw$, where the flow set is given by
	\begin{align*}
		\Cuw &:=\{(x, u_c, w_c) \in \reals^2 \times \reals \times [0, 1]:\\
		&|x| \geq 1, |x_1| \geq |u_c|, (|x|^2 -2)x_1^2 \leq u_cx_1 \leq (|x|^2 -1)x_1^2\}
	\end{align*}
	and jump map\footnote{
		$R(s) = \begin{bmatrix}
		\cos s & \sin s \\ -\sin s & \cos s
		\end{bmatrix}$ represents a rotation matrix.}
	$$\Guw(x, u_d, w_d) := \left\{ -R(u_dw_d) x, R(u_dw_d) x \right\},$$
	defined for every $(x, u_d, w_d) \in \Duw$, where the jump set is given by
	\begin{align*}
		\Duw : = \left \{(x, u_d, w_d) \in \reals^2 \times \reals \times [-1.1, 1.1]:
		x_1 = 0, |x| \geq 1, u_d \in \left[\frac{\pi}{4}, \frac{\pi}{2} \right] \right \}.
	\end{align*}
	\begin{figure}[h]
		\centering
		\includegraphics[width=0.6\unitlength]{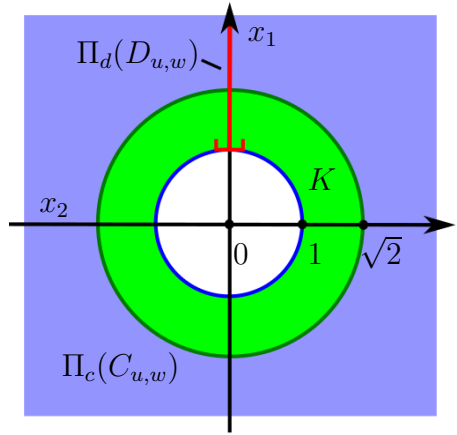}
		\put(-.28,.53){$x_1$}
		\put(-.56,.3){$x_2$}
		\put(-.3,.23){$0$}
		\put(-.2,.23){$1$}
		\put(-.13,.23){$\sqrt{2}$}
		\put(-.2,.33){$K$}
		\put(-.51,.5){$\Pd{\Duw}$}
		\put(-.53,.08){$\Pc{\Cuw}$}
		\caption{Set configuration for Example~\ref{ex:rcFI}.}
	\end{figure}

	\noindent Consider the set $K = \{x\in \reals^2: 1 \leq |x| \leq \sqrt{2} \},$ and a continuous state-feedback pair $(\kappa_c,\kappa_d)$ defined for every $x \in \reals^2$ as
	$$\kappa_c(x) = \left(|x|^2 - \frac{3}{2}\right)x_1 \text{, } \kappa_d(x) = \frac{\pi}{3}.$$
	By definition of $\Fuw$ and $\kc$, we have
	$$\Fcl(x, w_c) := \left\{ \begin{bmatrix}
	x_1^2 - \gamma\\ x_1x_2
	\end{bmatrix}\left(|x|^2 - \frac{3}{2}\right)x_1 w_c: \gamma \in [3,4]\right\},$$
	which is Lipschitz on the set $\Pc{\Cuw} \cap K = K$.
	The assumptions as well as conditions \ref{item:rcFI1} and \ref{item:rcFI4} in Corollary~\ref{coro:rcFI} hold by construction of $\Huw$, $\kpair$, and $K$.
	Consider a continuously differentiable function $V(x) := x^2_1+x^2_2$ for every $x \in \reals^2$.
	Since $\gamma \in [3,4]$ and $w_c \in [0,1]$, we have that for every $x$ such that $|x|  = 1$ and every $\xi \in \Fcl(x)$,
	\begin{align*}
		\inner{\nabla V(x)}{\xi} = 2x_1\xi_1 + 2x_2\xi_2 = (\gamma-1)x^2_1w_c \geq 0,
	\end{align*}
	and for every $x$ such that $|x|  = 2$ and every $\xi \in \Fcl(x)$,
	\begin{align*}
		\inner{\nabla V(x)}{\xi} = 2x_1\xi_1+2x_2\xi_2 = (2 -\gamma)x^2_1w_c \leq 0.
	\end{align*}
	Hence, item~\ref{item:rcFI3} holds and $\Lw = \emptyset$ by application of item~\ref{item:inner2} in Lemma~\ref{lem:innerTcone}.
	Condition \ref{item:rcFI2} holds because the rotation matrix $R$ only changes the direction of the vector $x$, while its magnitude remains the same after each jump.
	Item \ref{item:rcFI5} holds trivially as $\Lw = \emptyset.$
	Therefore, by an application of Corollary~\ref{coro:rcFI}, the set $K$ is robustly controlled forward invariant for system $\Huw$ via the given state-feedback pair $\kpair$.
	\hfill $\triangle$
\end{example}
}

%% file: Sections/4-RCFI-Lya.tex
For systematic invariance-based feedback design, we propose control Lyapunov functions that are tailored to forward invariance properties.
We refer to these functions as {\em robust control Lyapunov functions for forward invariance}.
Under appropriate conditions, these functions can be used to systematically design state-feedback laws that render a particular sublevel set robustly forward invariant.
In simple words, a robust control Lyapunov function for forward invariance, denoted as $\Ly$, allows to select the inputs of $\Huw$ as a function of the state $x$ so that a set of the form
	\begin{align}\label{eq:K}
		\Mr = \Ls \cap (\Pc{\Cuw} \cup \Pd{\Duw}),
	\end{align}
which is a subset of the $\level$-sublevel set of $\Ly$, has the robust controlled forward invariance property introduced in Definition~\ref{def:rcFI}.
As expected, and as formally stated next, the function $\Ly$ needs to satisfy certain CLF-like properties involving the constant $\level$ defining the level of the sublevel set $\Ls$ and the data of $\Huw$.
In its definition, we employ the set-valued map
\IfConf{
	\begin{align}\label{eq:rld}
		\Ld(x) : = \{u_d \in \uxd(x): & \Guw(x,u_d, \wxud(x, u_d)) \subset \notag \\
		& \qquad \Pc{\Cuw} \cup \Pd{\Duw}\},
	\end{align}
}{
	\begin{align}\label{eq:rld}
		\Ld(x) : = &\{u_d \in \uxd(x): \Guw(x,u_d, \wxud(x, u_d)) \subset \Pc{\Cuw} \cup \Pd{\Duw}\},
	\end{align}
}
for every $x \in \Pd{\Duw}$, which, at each such $x$, collects all inputs $u_d$ such that, regardless of the value of the disturbance, the state $x$ after jumps is in the projection of the flow and jump set to the state space, namely, in $\Pc{\Cuw} \cup \Pd{\Duw}$.

\begin{definition}(RCLF for forward invariance for $\Huw$)\label{def:rclf}
	Consider a hybrid system $\Huw = \datauw$ as in \eqref{eq:Huw}, a constant $\rU \in \reals$, and a continuous function $\Ly : \reals^n \rightarrow \reals$ that is also continuously differentiable on an open set containing $\Pc{\Cuw}$.
	Suppose there exist continuous functions $\rc : \reals^n \rightarrow \reals$ and $\rd : \reals^n \rightarrow \reals_{\geq 0}$ such that, for some $\level < \rU$, \NotConf{we have}
	\begin{align}
		& \rc(x) > 0 \quad \forall x \in \Ir, \label{eq:CLF-rC}\\
		& \rd(x) > 0 \quad \forall x \in \Ls. \label{eq:CLF-rD}
	\end{align}
	Then, the pair $\clfpair$ defines a {\em robust control Lyapunov function (RCLF) for forward invariance} of the sublevel sets of $\Ly$ for $\Huw$ if
	\IfConf{
	\begin{align}
		\begin{split}
			\inf_{u_c \in \uxc(x)} \sup_{w_c \in \wxuc(x, u_c)} \sup_{\xi \in \Fuw(x, u_c,w_c)}
			& \hspace{-2mm} \inner{\nabla \Ly(x)}{\xi} + \rc(x) \leq 0\\
			& \hspace{-6mm} \forall x \in \Ir \cap \Pc{\Cuw},
		\end{split}\label{eq:CLF-C}
	\end{align}
	\begin{align}
		\begin{split}
			\inf_{u_d \in \Ld(x)} \sup_{w_d \in \wxud(x, u_d)}\sup_{\xi \in \Guw(x, u_d,w_d)}
			& \hspace{-2mm} \Ly(\xi) + \rd(x) \leq \level\\
			& \hspace{-6mm} \forall x \in \Ls \cap \Pd{\Duw}.
		\end{split}\label{eq:CLF-D}
	\end{align}
	}{
	\begin{align}
		& \inf_{u_c \in \uxc(x)} \sup_{w_c \in \wxuc(x)} \sup_{\xi \in \Fuw(x, u_c,w_c)}
		\inner{\nabla \Ly(x)}{\xi} + \rc(x) \leq 0
		&\forall x \in \Ir \cap \Pc{\Cuw}, \label{eq:CLF-C}\\
		& \inf_{u_d \in \Ld(x)} \sup_{w_d \in \wxud(x)}\sup_{\xi \in \Guw(x, u_d,w_d)}
		\Ly(\xi) + \rd(x) \leq \level
		&\forall x \in \Ls \cap \Pd{\Duw}. \label{eq:CLF-D}
	\end{align}
}
\ConfSp{-6mm}

\hfill $\square$
\end{definition}

\vspace{-2mm}
\begin{remark}
Compared to a typical control Lyapunov function (see, e.g., \cite[Definition 2.1]{75}), the RCLF for forward invariance in Definition~\ref{def:rclf} is not constrained to be lower and upper bounded by class ${\cal K}_\infty$ functions relative to a set.
Note that \eqref{eq:CLF-C} does not impose conditions in the interior of $\Ls$, but to avoid $V(x)$ from being larger than $r$, \eqref{eq:CLF-D} is enforced on $x \in \Ls \cap \Pd{\Duw}$
The strict positivity requirements in \eqref{eq:CLF-rC} and \eqref{eq:CLF-rD} are essential to make continuous selections in the forthcoming result.
\end{remark}

\begin{remark}
	The definition of robust control Lyapunov function (RCLF) for forward invariance of the sublevel sets of $V$ in Definition~\ref{def:rclf} is related to the notion of barrier function and control barrier function.
	It should be noted that different barrier notions are proposed in the literature, for continuous-time \cite{wills2004barrier},  discrete-time \cite{prajna2005optimization}, and hybrid systems, including hybrid automata \cite{prajna2007framework} and hybrid inclusions \cite{183}.
	Some of these references present necessary and sufficient conditions for forward invariance; see, e.g., \cite{prajna2005necessity} and  \cite{wisniewski2015converse}.
	With such barrier functions typically denoted as $B$,  the problem of rendering an $\level$-sublevel of a function $\Ly$ studied in this paper naturally leads to the barrier function $B(x) = \Ly(x) - \level$.
	With such definition, the barrier function resulting from this construction is close to the definition in \cite{prajna2007framework}.
	In particular, for a hybrid system with inputs and disturbances, our results allow for the design of control laws that guarantee robust forward invariance of the set $\{x : B(x) \leq 0\}$, properly restricted to the union of the flow and jump set.
\end{remark}

Next, we illustrate the concept of RCLFs for forward invariance in Definition~\ref{def:rclf} for the robotic manipulator system introduced in Example~\ref{ex:hmass-model}.
\begin{example}(RCLF for forward invariance for the robotic manipulator)\label{ex:hmass-rclf}
		Consider the function
		\begin{align}\label{eq:K-arm}
			\Ly(x) = \frac{1}{2}x^\top P x, \text{\quad with }\quad P = \begin{bmatrix} a & c\\ c &b \end{bmatrix} > 0.
		\end{align}
		We define the safe region described in Example~\ref{ex:hmass-model} using the $r$-sublevel set of $\Ly$, i.e., $L_\Ly(\level)$ with $\level > 0$.\NotConf{\footnote{
			The set $L_\Ly(\level)$ is an ellipse in $\reals^2$ such that, after the input is assigned to a state-feedback law, it is robust controlled forward invariant.}}
		Since $\Pc{\Cu} \scriptsize{\cup} \Dw \scriptsize{=} \reals^2$, the control objective is achieved by rendering the set
		\begin{align}\label{eq:raK}
			\K = \Ls \cap (\Pc{\Cu} \cup \Pd{\Dw}) = \Ls
		\end{align}
		robustly controlled forward invariant for $\Huw$.
	Considering the state-feedback control law given by $u_c = -Kx$ with $K = [k_p \ k_d]$, for every $x \in \Pc{\Cu}$ with $k_p, k_d > 0$.
	By properly designing $K$, we aim to render the set $\K$ given in \eqref{eq:raK} robustly controlled forward invariant for $\Huw$ in Example~\ref{ex:hmass-model}.
	To this end, under the effect of this feedback, the (set-valued) flow map can be written as 
	\begin{align}
		F_k(x) := \begin{bmatrix} 0 & 1\\ -k_p - k_c & A(x) \end{bmatrix}x,
	\end{align}
	where $A(x) :=
	\begin{cases}
		-k_d & \ifeq x_1 > 0\\
		-k_d - \overline{\con}\{ 0, b_c\} & \ifeq x_1 = 0\\
		-k_d-b_c & \ifeq x_1 < 0.
	\end{cases}$\\
	\noindent Using $\Ly$ defined in \eqref{eq:K-arm}, for every $x \in \reals^2$ and every $\eta \in F_k(x)$, if $\frac{a}{c} \geq \frac{k_c}{b_c}$, we have $\inner{\nabla \Ly(x)}{\eta} \leq x^\top Qx,$ where
	$Q = \begin{bmatrix} -2ck_p & a - bk_p - ck_d\\ a - bk_p - ck_d & 2c - 2bk_d \end{bmatrix}.$
	If we chose feedback parameters such that
	\begin{align}
		&4bck_pk_d - 4c^2k_p > (a - bk_p - ck_d)^2, \label{eq:detQ}\\
		&k_p > 1 - \frac{b}{c}k_d \label{eq:trQ}
	\end{align}
	then, the matrix $Q$ is negative definite\IfConf{.}{(see details in Lemma and proof in report version of this paper).}
	\NotConf{
		More precisely, the determinant of $Q(x)$, i.e., $\min\limits_{\xi \in F_k(x)} \det Q(x) = (4cbk_pk_d - 4c^2k_p) - (a - bk_p - ck_d)^2$, is strictly positive because of \eqref{eq:detQ}, and the trace of $Q$, i.e., $\max\limits_{\xi \in F_k(x)}\text{tr} Q(x) = 2c - 2ck_p - 2bk_d$, is strictly negative because of \eqref{eq:trQ}.
}
	Let $\level < \rU = b\overline{v}^2$ and $\rc(x) = - x^\top Q x,$ for every $x \in \reals^2$, \eqref{eq:CLF-C} holds since when, in particular, $u_c = kx$ we obtain
	$\inner{\nabla \Ly(x)}{\eta} + \rc(x) \leq 0.$
	Then, for every $x \in \reals^2$ and every $\eta \in F_k(x)$.
	In addition, given $\level \in \left(\frac{b\overline{v}^2}{2}, b\overline{v}^2 \right)$ we consider $\rd(x) : = \frac{(1 - e_2^2)b\overline{v}^2}{2}.$
	Hence, for every $x \in \Ls \cap D$, we have
	\IfConf{
	\begin{align*}
		\max\limits_{w_d \in [e_1, e_2]} &\Ly(G_w(x,w_d)) + \rd(x) - \level
		= \left(\frac{a}{2} x_1^2 + cx_1x_2+ \frac{b}{2}x_2^2\right) \\
		& \hspace{10mm} - \level +  \frac{(1 - e_2^2)b(\overline{v}^2 -x_2^2)}{2} - (1-e_2)cx_1x_2,
	\end{align*}
	}{
	\begin{align*}
		\max\limits_{w_d \in [e_1, e_2]} &\Ly(G_w(x,w_d)) + \rd(x) - \level\\
		= &\left(\frac{a}{2} x_1^2 + ce_2x_1x_2+ \frac{b}{2}(e_2x_2)^2\right) +  \frac{(1 - e_2^2)b\overline{v}^2}{2} - \level\\
		= &\left(\frac{a}{2} x_1^2 + cx_1x_2+ \frac{b}{2}x_2^2\right) - \level +  \frac{(1 - e_2^2)b(\overline{v}^2 -x_2^2)}{2}\\
		& \hspace{5cm} - (1-e_2)cx_1x_2,
	\end{align*}
}
	which is nonpositive since $e_2 \in(0,1), x_1 >0, x_2 \geq \overline{v}$ and every $x \in \Ls$ is such that $\Ly(x) \leq \level$.
	Therefore, \eqref{eq:CLF-D} holds and the pair $(\Ly, \rU)$ defines a robust control Lyapunov function for forward invariance for $\Huw$.
	\hfill $\triangle$
\end{example}

\ConfSp{-3mm}
Given a pair $\clfpair$ defined as in Definition~\ref{def:rclf} for $\Huw$ and $\level < \rU$ satisfying the conditions therein, our approach consists of selecting a state-feedback law pair $\kpair$ from these inequalities.
In fact, we are interested in synthesizing a pair $\kpair$ that, in particular, satisfies
\IfConf{
	\begin{align}
		\sup_{w_c \in \wxuc(x, \kc(x))} \sup_{\xi \in \Fuw(x,\kc(x),w_c)}
		& \inner{\nabla \Ly(x)}{\xi} + \rc(x) \leq 0\\
		& \hspace{3mm} \forall x \in \Ir \cap \Pc{\Cuw},\\ \label{eq:CLF-C-WithFeedback}
		\sup_{w_d \in \wxud(x, \kd(x))}\sup_{\xi \in \Guw(x, \kd(x),w_d)}
		& \Ly(\xi) + \rd(x) \leq \level\\
		& \hspace{5mm} \forall x \in \Ls \cap \Pd{\Duw}. \label{eq:CLF-D-WithFeedback}
	\end{align}
}{
	\begin{align}
		& \sup_{w_c \in \wxuc(x, \kc(x))} \sup_{\xi \in \Fuw(x,\kc(x),w_c)}
		\inner{\nabla \Ly(x)}{\xi} + \rc(x) \leq 0
		&\forall& x \in \Ir \cap \Pc{\Cuw}, \label{eq:CLF-C-WithFeedback}\\
		& \sup_{w_d \in \wxud(x, \kd(x))}\sup_{\xi \in \Guw(x, \kd(x),w_d)}
		\Ly(\xi) + \rd(x) \leq \level
		&\forall& x \in \Ls \cap \Pd{\Duw}, \label{eq:CLF-D-WithFeedback}
	\end{align}
}
Under certain mild conditions, such a pair renders the set $\Mr$ in \eqref{eq:K} robustly controlled forward invariant for $\Huw$.
Interestingly, with a constant parameter $\sigma \in (0,1)$, the selection of such a feedback pair can be performed by defining sets that nicely depend on the functions
	\begin{align}\label{eq:GammaC}
	\IfConf{
		\begin{split}
			&\Gamma_c(x, u_c) : = \\
			&\quad \begin{cases}
				\sup\limits_{w_c \in \wxuc(x, u_c)} \sup\limits_{\xi \in \Fuw(x, u_c, w_c)}
				\inner{\nabla \Ly(x)}{\xi} + \sigma\rc(x)\\
				& \hspace{-2cm}\ifeq (x,u_c) \in \Delta_c,\\
				-\infty &\hspace{-2cm}\otherw
			\end{cases}
		\end{split}
		}{
		\Gamma_c(x, u_c) : = \begin{cases}
			\sup\limits_{w_c \in \wxuc(x, u_c)} \sup\limits_{\xi \in \Fuw(x, u_c, w_c)}
				\inner{\nabla \Ly(x)}{\xi} + \sigma\rc(x)
				&  \ifeq (x,u_c) \in \Delta_c,\\
			-\infty &\otherw
		\end{cases}
	}
	\end{align}
for each $(x, u_c, w_c) \in \dc$, and
	\begin{align}\label{eq:GammaD}
	\IfConf{
		\begin{split}
			&\Gamma_d(x, u_d) : =\\
			&\quad \begin{cases}
				\sup\limits_{w_d \in \wxud(x, u_d)} \sup\limits_{\xi \in \Guw(x, u_d, w_d)}
				\Ly(\xi) + \sigma\rd(x) - \level\\
				& \hspace{-2cm} \ifeq (x,u_d) \in \Delta_d,\\
				-\infty & \hspace{-2cm} \otherw,
			\end{cases}
		\end{split}
		}{
			\Gamma_d(x, u_d) : =
				\begin{cases}
					\sup\limits_{w_d \in \wxud(x, u_d)}  \sup\limits_{\xi \in \Guw(x, u_d, w_d)} \!\!\!\!
					\Ly(\xi) + \sigma\rd(x) - \level & \ifeq (x,u_d) \in \Delta_d,\\
					-\infty & \otherw,
				\end{cases}
		}
	\end{align}
for each $(x,u_d, w_d) \in \dd$,
where $\Delta_c:=\{(x,u_c): (x, u_c,w_c) \in (\Mc \times \Uc \times \Wc) \cap \Cuw\}$, 
$\Delta_d:=\{(x,u_d): (x, u_d,w_d) \in (\Md \times \Ud \times \Wd) \cap \Duw\}$.
Moreover, we define
\IfConf{
	\begin{align}\label{eq:Ms}
		\Mc \scriptsize{: =} \Ir \cap \Pc{\Cuw}, \ \Md \scriptsize{: =} \Ls \cap \Pd{\Duw}. \quad
	\end{align}
}{
\begin{align}
	\begin{split}
		&\Mc : = \Ir \cap \Pc{\Cuw},\\
		&\Md : = \Ls \cap \Pd{\Duw}.
	\end{split}
	\label{eq:Ms}
\end{align}
}
\noindent
In fact, with these functions defined, by introducing  the set-valued maps
$\{u_c \in \uxc(x): \Gamma_c(x, u_c) < 0 \},$ and $\{u_d \in \Ld(x): \Gamma_d(x, u_d) < 0 \}$
which are the so-called {\em regulation maps} \cite{robust}, our approach is to determine a  state-feedback pair $\kpair$ that is selected from these maps; i.e., $\kpair$ is such that
\IfConf{
	$\kc(x) \in \{u_c \in \uxc(x): \Gamma_c(x, u_c) < 0 \},$ and
	$\kd(x) \in \{u_d \in \Ld(x): \Gamma_d(x, u_d) < 0 \}$
}{
$$\kc(x) \in \{u_c \in \uxc(x): \Gamma_c(x, u_c) < 0 \},$$
$$\kd(x) \in \{u_d \in \Ld(x): \Gamma_d(x, u_d) < 0 \}$$
}
at the appropriate values of the state $x$.

In Section~\ref{sec:rcFI-Lya}, we provide key results on robust forward invariance of sublevel sets of CLF-like functions, which are used in our CLF approach.
It turns out that when an RCLF for forward invariance for $\Huw$ is provided, regulation maps as outlined above can be constructed for selecting a state-feedback satisfying the conditions in the forthcoming Theorem~\ref{prop:Ly_pre} and Theorem~\ref{prop:Ly}; hence, the results in Section~\ref{sec:rcFI-Lya} enable us to show the desired invariance property under feedback.
Since according to Lemma~\ref{lem:hbc}, the closed-loop system $\Hcl$ satisfies conditions \ref{item:A1}-\ref{item:A3} in Definition~\ref{def:hbc} when the applied state-feedback pair is continuous,
we seek the design of a state-feedback pair $\kpair$ with $\kc$ and $\kd$ being continuous functions of
the state.
For this purpose, in Section~\ref{sec:rclf-FI}, we first reveal conditions assuring the existence of continuous selections from the regulation maps.
Our main design results are in Section~\ref{sec:rselect}, where we provide a explicit construction of $\kpair$ with pointwise minimum norm.

\ConfSp{-4mm}
\subsection{Robust Forward Invariance of Sublevel sets of Lyapunov-like Functions}
\label{sec:rcFI-Lya}
Building from \cite[Section V]{PartI-Arxiv}, we provide conditions for robust forward (pre-)invariance of sublevel sets of $\Ly$ for $\Hcl$, which in turn, provide insight for the invariance-based control design methods in Section~\ref{sec:rclf-FI} and Section~\ref{sec:rselect}.
More precisely, given a function $\Ly:\reals^n \rightarrow \reals$, we derive sufficient conditions to render its $\level-$sublevel set, with some abuse of notation, given as
\begin{align}\label{eq:Mr}
	\K= \Ls \cap (\Pwc{\Ccl} \cup \Pwd{\Dcl})
\end{align}
robust controlled forward (pre-)invariant for $\Huw$.

We consider Lyapunov-like functions that are tailored to forward invariance as introduced in Definition~\ref{def:rclf}.
Unlike the case for asymptotic stability, the proposed Lyapunov candidate does not necessarily strictly decreases along solutions outside of $\K$ or is nonincreasing inside of $\K$.
Building from \cite[Theorem 5.1]{PartI-Arxiv}, the next result characterizes the robust forward pre-invariance of $\K$ in terms of a Lyapunov-like functions.
\IfConf{
	Proof for Theorem~\ref{prop:Ly_pre} is presented in \cite{PartII-Arxiv}.
}{}
\begin{theorem}(robust forward pre-invariance of $\K$)\label{prop:Ly_pre}
	Given a hybrid system $\Hcl = \datacl$ as in \eqref{eq:Hcl}, suppose there exist a constant $\rU \in \reals$ and a continuous function $\Ly:\reals^n \rightarrow \reals$ that is continuously differentiable on an open set containing $\Pwc{\Ccl}$ such that
	\IfConf{
		\begin{align}
			\langle \nabla \Ly(x), \eta \rangle \leq 0 \qquad \forall (x, w_c) &\in (\Ir \times \Wc) \cap \Ccl,\\
			\forall \eta &\in \Fcl(x,w_c), \label{eq:ly1}
		\end{align}
		\begin{align}
			\Ly(\eta) \leq \level \qquad \ \qquad \forall (x, w_d) &\in (\Ls \times \Wd) \cap \Dcl,\\
			\forall \eta &\in \Gcl(x,w_d), \label{eq:ly2} \\
	\end{align}
}{
		\begin{align}
	\langle \nabla \Ly(x), \eta \rangle \leq 0 \qquad \forall (x, w_c) &\in (\Ir \times \Wc) \cap \Ccl,
	\forall \eta \in \Fcl(x,w_c), \label{eq:ly1}
\end{align}
\begin{align}
	\Ly(\eta) \leq \level \qquad \ \qquad \forall (x, w_d) &\in (\Ls \times \Wd) \cap \Dcl,
	\forall \eta \in \Gcl(x,w_d), \label{eq:ly2} \\
\end{align}
}

\ConfSp{-6mm}
\noindent for some $\level \scriptsize{\in} (-\infty,\rU)$ such that $\K$ is nonempty and closed, and
	\begin{align}\label{eq:lyJump}
		\Gcl((\K \times \Wd) \cap \Dcl) \subset \Pwc{\Ccl} \cup \Pwd{\Dcl}
	\end{align}
	holds.
	Then, the set $\K$ is robustly forward pre-invariant for $\Hcl$.
\end{theorem}

\NotConf{
	The proof of Theorem~\ref{prop:Ly_pre} is presented in Appendix~\ref{appendix:ProofOfprop:Ly_pre}.
}

Conditions \eqref{eq:ly1}, \eqref{eq:ly2} and \eqref{eq:lyJump} can be used to check whether an already designed state-feedback pair $\kpair$ renders $\K$ given as in \eqref{eq:Mr} robustly controlled forward invariant for $\Huw$.

\begin{remark}
	A typical set of Lyapunov conditions for asymptotic stability analysis can be found in \cite[Theorem 3.18]{65}.
	These conditions ensure the decrease of $\Ly$ along solutions that are initialized outside of $\A$.
	In comparison to Theorem~\ref{prop:Ly_pre}, forward invariance requires the properties of the data of $\Hw$ and of $V$ relative to the set of interest, in our case, $\K$.
	Compared to \cite[Definition 3.16]{65} and \cite[Theorem 3.18]{65}, a function $\Ly$ as in Theorem~\ref{prop:Ly_pre} is a Lyapunov function candidate that satisfies less restrictive conditions, and certainly, does not guarantee attractivity.
	Such function $\Ly$ is neither bounded (from below and above) by two class-$\classKinfty$ functions, namely, it does not need to be positive definite and radially unbounded, nor has its change along solutions bounded by a negative definite function of the distance to the set of interest.
	In particular, for stability in the nominal case, item (3.2b) in \cite[Theorem 3.18]{65} asks $\inner{\nabla \Ly(x)}{\eta} \leq 0$ for all $x \in \Lss \cap C$ and all $\eta \in F(x)$, while \eqref{eq:ly1} allows $\inner{\nabla \Ly(x)}{\eta}$ to be positive at points $x \in \inter \Ls \cap C$.
	Similarly, during jumps, item (3.2c) in \cite[Theorem 3.18]{65} demands the change $\Ly(\eta) - \Ly(x)$ to be nonpositive for every $x \in \Ls\cap D$; while \eqref{eq:ly2} allows such changes to be positive at points $x \in \inter \Ls \cap D$ as long as it is such that $\Ly(\eta) \leq \level.$
	Such properties ensure solutions stay within $\Ls$ for any qualifying $\level < \rU$.\footnote{Note that solution pairs may escape $\Ls$ when $\level = \rU$.
		This is because $\inner{\nabla \Ly(x)}{\eta}$ is allowed to be zero in \eqref{eq:ly1}.}
	Note that \eqref{eq:ly1} and \eqref{eq:ly2} do not imply that maximal solutions are complete, neither to $\Hcl$ nor to the restriction of $\Hcl$ to $\Lss$.
	\NotConf{
	Other alternative conditions may involve a locally Lipschitz flow map $\Fcl$ similar to Corollary~\ref{coro:rcFI}.
}
\end{remark}

\begin{remark}\label{rem:zeroV}
	It is worth noting that due to being inequalities, the conditions in Theorem~\ref{prop:Ly_pre} cover the special cases where $\Ly$ remains constant on the flow set or on the jump set.
	In such a case, \eqref{eq:ly1} and \eqref{eq:ly2} in Theorem~\ref{prop:Ly_pre} are given by
	\IfConf{
		\begin{align}
			\langle \nabla \Ly(x), \eta \rangle  = 0 \qquad \ \  \forall  (x, w_c) &\in (\Lss \times \Wc) \cap \Ccl,\\
			\forall \eta &\in \Fcl(x,w_c), \label{eq:Rly1}
		\end{align}
		\begin{align}
			\Ly(\eta) -  \Ly(x) = 0 \qquad \forall (x, w_d) &\in (\Ls \times \Wd) \cap \Dcl,\\
			\forall \eta &\in \Gcl(x,w_d), \label{eq:Rly2}
		\end{align}
}{
	\begin{align}
	&\langle \nabla \Ly(x), \eta \rangle = 0 &\forall & (x, w_c) \in (\Lss \times \Wc) \cap \Ccl, \eta \in \Fcl(x,w_c),
	\label{eq:Rly1}\\
	&\Ly(\eta) - \Ly(x) = 0 & \forall & (x, w_d) \in (\Ls \times \Wd) \cap \Dcl, \eta \in \Gcl(x,w_d), \label{eq:Rly2}
	\end{align}
}
	respectively.
	Intuitively, when $\Ly$ does not change on $\Lss$, for any $\level < \rU$, solution pairs to $\Hcl$ stay within the $\level-$sublevel set during flows and jumps.
	Namely, we can employ \eqref{eq:Rly1} and \eqref{eq:ly2}, or \eqref{eq:ly1} and \eqref{eq:Rly2}, to verify robust forward pre-invariance of $\K$.
\end{remark}

\ConfSp{-1mm}
The observations in Remark~\ref{rem:zeroV} also extend to the case of hybrid systems where the control inputs affect only one regime, namely, either the flows or the jumps and $\Ly$ does not increase during the regime that is not affected by inputs.
Consequently, when verifying a RCLF candidate for such systems, we can omit checking the condition in \eqref{eq:CLF-C} if \eqref{eq:Rly1} or \eqref{eq:ly2} holds (or, respectively, omit checking \eqref{eq:CLF-D} when \eqref{eq:ly1} or \eqref{eq:Rly2} holds).
One such example is the controlled single-phase DC/AC inverter system in \cite[Section VI]{PartI-Arxiv}, for which \eqref{eq:Rly2} holds (a special case of \eqref{eq:ly2}).
Another example is the bouncing ball system introduced in Example~\ref{ex:bball-model}, where the total energy of the ball is used to construct the function $\Ly$ for invariance analysis.
During flows, no energy loss is considered.
Hence, the total energy level of the system remains constant during flows, which implies that the special case of \eqref{eq:ly1}, namely \eqref{eq:Rly1}, holds.
We illustrate such concept in the following example.

	\begin{example}(The RCLF for forward invariance for the bouncing ball system)\label{ex:bball-rclf}
		We define $\Ly(x) : = -E(x)$ for every $x \in C\cup \Pd{\Duw}$.
		Following formula given in \eqref{eq:K}, the control objective is achieved by rendering the set
		\begin{align}\label{eq:BK}
		\K = L_\Ly(-\gamma \hmin) \cap (C \cup \Pd{\Duw})
		\end{align}
		robustly controlled forward invariant for $\Huw$.
		Given system parameters $e_1, e_2, e_p, \vmax$ and $\hmax$, the control goal can be achieved for $\hmin$ such that $\sqrt{\gamma\left(\hmin + \frac{\varepsilon}{2}\right)} < e_1 \sqrt{\Emax}$ and with $\varepsilon > 0$,
		\IfConf{
			\vspace{-1mm}
		\begin{align}\label{eq:epsilon}
			\gamma(\hmin + \varepsilon) \leq 0.5(1 +e_1 - e_2)^2\Emax.
		\end{align}
	}{
		\begin{align}\label{eq:epsilon}
		\gamma(\hmin + \varepsilon) \leq \frac{(1 +e_1 - e_2)^2}{2}\Emax.
		\end{align}
	}
		Since the control input appears in the map $G_1$ only, for every $x \in \Pd{D_1}$, according to \eqref{eq:rld}, the set $\Ld$ in \eqref{eq:rld} is given by
		\ConfSp{-1mm}
		\begin{align}\label{eq:BLd}
		\Ld(x) = [0, \sqrt{2\Emax} + e_2x_2].
		\end{align}
		In fact, given such $x$, $\Ld$ collects all control input values $u_d$ such that $G_1(x,u_d, w_d) \in C \cup \Pd{\Duw}$ for all $w_d \in [e_1, e_2]$; i.e., every such $u_d$ is such that $E(0, G_1(x, u_d, e_2)) \leq \Emax$.
		
		Now, consider the constant $\rU = -\gamma(\hmin - \varepsilon)$ and the function $\rd$ defined as $\rd(x)  = \gamma \varepsilon$ for every $x \in \Ls$.
		We show that the pair $\clfpair$ defines a RCLF for forward invariance as in Definition~\ref{def:rclf}.
		First, \eqref{eq:Rly1} holds on $C$  since, for every $x \in C$,
		\IfConf{
			$\inner{\nabla \Ly(x)}{F(x)} = -x_2(-\gamma) - \gamma x_2 = 0.$
		}{
		\begin{align}\label{eq:BVdot}
		\inner{\nabla \Ly(x)}{F(x)} = -x_2(-\gamma) - \gamma x_2 = 0.
		\end{align}
	}
		Then, we show the pair $\clfpair$ is such that \eqref{eq:CLF-D} holds for $\level = -\gamma\hmin < \rU$.
		Moreover, for every $x \in \Ls \cap \Pd{D_1}$, we have
		\IfConf{
			\begin{align*}
				\min\limits_{u_d \in \Ld(x)} \max\limits_{w_d \in [e_1,e_2]} & \Ly(G_1(x,u_d, w_d))\\
				&= \min\limits_{u_d \in \Ld(x)} \max\limits_{w_d \in [e_1,e_2]} \left\{- 0.5(u_d - w_dx_2)^2\right\}\\
				&= - 0.5(\sqrt{2\Emax} + e_2x_2 - e_1x_2)^2
			\end{align*}
		}{
		\begin{align*}
		\min\limits_{u_d \in \Ld(x)} \max\limits_{w_d \in [e_1,e_2]} & \Ly(G_1(x,u_d, w_d))\\
		&= \min\limits_{u_d \in \Ld(x)} \max\limits_{w_d \in [e_1,e_2]} \left\{- \frac{(u_d - w_dx_2)^2}{2} \right\}\\
		&= - \frac{(\sqrt{2\Emax} + e_2x_2 - e_1x_2)^2}{2}.
		\end{align*}
	}
		Since $x_2 \in [-\sqrt{2\Emax}, -\sqrt{2\gamma\hmin}]$ and due to condition \eqref{eq:epsilon}, we have
		\IfConf{
			\begin{align*}
				\min\limits_{u_d \in \Ld(x)} &\max\limits_{w_d \in [e_1,e_2]}  \Ly(G_1(x,u_d, w_d)) + \rd(x)\\
				&\leq  - 0.5(\sqrt{2\Emax} + (e_2 - e_1)(-\sqrt{2\Emax}))^2 + \rd(x)\\
				&= - 0.5(1+e_1-e_2)^2 \Emax + \gamma \varepsilon
				\leq -\gamma \hmin = \level
			\end{align*}
		}{
		\begin{align*}
		\min\limits_{u_d \in \Ld(x)} &\max\limits_{w_d \in [e_1,e_2]}  \Ly(G_1(x,u_d, w_d)) + \rd(x)\\
		&\leq  - \frac{(\sqrt{2\Emax} + (e_2 - e_1)(-\sqrt{2\Emax}))^2}{2} + \rd(x)\\
		&= - \frac{(1+e_1-e_2)^2}{2}\Emax + \gamma \varepsilon
		\leq -\gamma \hmin = \level.
		\end{align*}
	}
		For every $x \in \Ls \cap \Pd{D_2}$, we have $x_2 \in [0, \vmax]$ and
		\IfConf{
			\begin{align*}
				\min\limits_{u_d \in \Ld(x)} \max\limits_{w_d \in [e_1,e_2]} & \Ly(G_2(x))\\
				&= - 0.5(\min\{- \ep x_2, -\delta_p\})^2 -\gamma \hmax < \level.
			\end{align*}
		}{
			\begin{align}\label{eq:Bg2}
			\hspace{-5mm}
			\begin{split}
			\min\limits_{u_d \in \Ld(x)} \max\limits_{w_d \in [e_1,e_2]}  \Ly(G_2(x))
			= - \frac{(\min\{- \ep x_2, -\delta_p\})^2}{2} -\gamma \hmax < \level.
			\end{split}
			\end{align}
		}
		Hence, the pair $\clfpair$ defines a robust control Lyapunov function for forward invariance for $\Huw$ according to Remark~\ref{rem:zeroV} and Definition~\ref{def:rclf}.
		\hfill $\triangle$
	\end{example}

Next, we derive conditions rendering the set $\K \subset \reals^n$ in \eqref{eq:Mr} robustly forward invariant for $\Hcl$ given as in \eqref{eq:Hcl}.
According to Definition~\ref{def:rcFI}, these conditions also imply the robustly controlled forward invariance of $\K$ for $\Huw$ via the pair $\kpair$.
The next result, whose proof is in Appendix~\ref{appendix:ProofOfprop:Ly}, follows from \cite[Theorem 5.1]{PartI-Arxiv} and ensures that every solution pair $(\phi, w) \in \sol_{\Hcl}(\K)$ has $\rge \phi \subset \K$.
Moreover, the proposed set of conditions guarantee existence and completeness of maximal solution pairs to $\Hcl$ from $\K$.
\begin{theorem}(robustly forward invariance of $\K$)\label{prop:Ly}
	Given a hybrid system $\Hcl =$ $\datacl$ as in \eqref{eq:Hcl}, suppose the set $\Ccl$ is closed, item~\ref{item:A2} in Definition~\ref{def:hbc} holds, and $(x, 0) \in \Cw$ for every $x \in \Pwc{\Ccl}$.
	Suppose there exist a constant $\rU \in \reals$ and a continuous function $\Ly:\reals^n \rightarrow \reals$ that is continuously differentiable on an open set containing $\Pwc{\Ccl}$ such that \eqref{eq:ly1} and \eqref{eq:ly2} in Theorem~\ref{prop:Ly_pre} hold for some $\level \in (-\infty, \rU)$ such that $\K$ is nonempty and closed, and \eqref{eq:lyJump} in Theorem~\ref{prop:Ly_pre} holds.
	Moreover, suppose
	\begin{enumerate}[label = \ref{prop:Ly}.\arabic*), leftmargin=3.2\parindent]
		\item\label{item:Ly1} for every $x \in \Ly^{-1}(\level) \cap \Pwc{\Ccl}$, $\nabla \Ly(x) \neq 0$;
		\item\label{item:Ly2} for every $x \in (\Ls \cap \partial \Pwc{\Ccl})\setminus \Pwd{\Dcl}$, $\Fcl(x, 0) \cap T_{\Pwc{\Ccl}}(x) \neq \emptyset$;
		\item\label{item:Ly3} for every $x \in (\Ly^{-1}(\level) \cap \partial \Pwc{\Ccl})\setminus \Pwd{\Dcl}$, the set $\Xi_x : = \{ \xi \in \Fcl(x, 0) \cap T_{\Pwc{\Ccl}}(x) :  \inner{\nabla \Ly(x)}{\xi} < 0\}$ is nonempty;
		\item\label{item:Ly4} $(\K \times \Wc) \cap \Ccl$ is compact, or
		$\Fcl$ has linear growth on $(\K \times \Wc) \cap \Ccl$.
	\end{enumerate}
	Then, the set $\K$ is robustly forward invariant for $\Hcl$.
\end{theorem}
The proof of Theorem~\ref{prop:Ly} is presented in Appendix~\ref{appendix:ProofOfprop:Ly}.

Compared to \cite[Theorem 5.1]{PartI-Arxiv}, item~\ref{item:Ly3} does not require the set $\Pwc{\Cw}$ to be regular as in item~5.1.3) of \cite[Theorem 5.1]{PartI-Arxiv}; see also Lemma~\ref{lem:Tcone} for details.

\begin{remark}
	Forward invariance that is uniform in the disturbances is key for certifying safety in real-world applications.
	As mentioned in Section~\ref{sec:intro}, barrier certificates have been shown to be useful for the study of safety, i.e., the problem of whether solutions initiated from a given set would reach an unsafe set.
	In particular, \cite{prajna2004safety} and \cite{kong2013exponential} pertain to safety for a class of hybrid systems modeled as hybrid automata.
	In these articles, barrier functions are used to characterize safe sets
	A barrier function has strictly positive values in the unsafe sets and nonpositive values otherwise.
	The conditions proposed guarantee that along every solution from an initial set, the values of these functions are nonincreasing.
	When compared to the conditions in \cite{prajna2004safety} and \cite{kong2013exponential}, the control Lyapunov function for forward invariance in Theorem~\ref{prop:Ly} does not need to be strictly positive outside of the set to be rendered forward invariant, c.f. in \cite[Theorem 2]{prajna2004safety};
	nor does need to satisfy the exponential condition required in \cite[Theorem 1]{kong2013exponential}.
	For nonlinear continuous-time system, \cite{ames2016control} provides two types of control barrier functions and compares them to exponentially stabilizing control Lyapunov functions.
	Aside from the differences in signs within the set of interests and the type of systems we study, our results do not require the control input to be locally Lipschitz as in \cite[Corollary 1]{ames2016control}; see, e.g., Theorem~\ref{prop:Ly}.
\end{remark}




%% file: Sections/5-RCLF-FI.tex
Next, building from Theorem~\ref{prop:Ly_pre}, we establish conditions to guarantee existence of a continuous state-feedback pair $\kpair$ to render the set $\Mr$ robustly controlled forward pre-invariant for $\Huw$.
\begin{theorem}(existence of state-feedback pair for robust controlled forward pre-invariance using RCLF for forward invariance) \label{thm:weakExist}
	Consider a hybrid system $\Huw = \datauw$ as in \eqref{eq:Huw} satisfying conditions \ref{item:A1p}-\ref{item:A3p} in Lemma~\ref{lem:hbc} and such that $\wxuc$ and $\wxud$ are locally bounded.
	Suppose there exists a pair $\clfpair$ that defines a robust control Lyapunov function for forward invariance for $\Huw$ as in Definition~\ref{def:rclf}.
	Let $\level < \rU$ satisfy \eqref{eq:CLF-rC}-\eqref{eq:CLF-D}, $\Ld$ be given as in \eqref{eq:rld}, and $\sigma \in (0,1)$.
	If the following conditions hold:
	\begin{enumerate}[label = \ref{thm:weakExist}.\arabic*),leftmargin=3.2\parindent]
		\item\label{item:weakExist1} The set-valued maps $\uxc$ and $\Ld$ are lower semicontinuous, and
		$\uxc$ and $\Ld$ have nonempty, closed, and convex values on the sets $\Pc{\Cuw}$ and $\Md$ as in \eqref{eq:Ms}, respectively;
		\item\label{item:weakExist2} For each $x \in \Mc$, the function $u_c \mapsto \Gamma_c(x, u_c)$ in \eqref{eq:GammaC} is convex on $\uxc(x)$ and,
		for each $x \in \Md$, the function $u_d \mapsto \Gamma_d(x, u_d)$ in \eqref{eq:GammaD} with is convex on $\Ld(x)$;
	\end{enumerate}
	then, the set $\Mr$ in \eqref{eq:Mr} is robustly controlled forward pre-invariant for $\Huw$ via a state-feedback pair $\kpair$ with $\kc$ being continuous on $\Mc$ and $\kd$ being continuous on $\Md$.
\end{theorem}

\begin{proof}
	To establish the result, we first show the existence of continuous control laws for a restricted version of the original hybrid system $\Huw$ that is given by
\ConfSp{-1mm}
	\begin{align}\label{eq:rHuw}
		\widetilde{\HS}\uw
		\begin{cases}
			\dot{x} \ \ \in \Fuw(x, u_c, w_c) & (x,u_c, w_c) \in \wC\uw\\
			x^+ \in \Guw(x,u_d, w_d)& (x,u_d, w_d) \in \wD\uw,
		\end{cases}
	\end{align}
	where $\wC\uw := (\Mc \times \Uc \times \Wc) \cap \Cuw$ and $\wD\uw := (\Md \times \Ud \times \Wd) \cap \Duw.$
	To this end, using $\Gamma_c$ and $\Gamma_d$ given as in \eqref{eq:GammaC} and \eqref{eq:GammaD}, for each $x \in \reals^n$, we define the set-valued maps
	\begin{align}
		\begin{split}
			& \wS_c(x) : = \{u_c \in \uxc(x): \Gamma_c(x, u_c) < 0 \},\\
			& \wS_d(x) : = \{u_d \in \Ld(x): \Gamma_d(x, u_d) < 0 \}.
		\end{split}\label{eq:Sstar}
	\end{align}
	By definition of $\Ld(x)$ in \eqref{eq:rld} and condition \ref{item:weakExist1}, the maps $\uxc$ and $\Ld$ are lower semicontinuous and for every $x \in \Md, \Ld(x)$ is a nonempty, convex subset of $\uxd(x)$.
	Then, we show the maps $\wS_c$ and $\wS_d$ are lower semicontinuous by applying Corollary~\ref{coro:2.13}.
	First, we establish that the functions $\Gamma_c$ and $\Gamma_d$ are upper semicontinuous by observing the properties of the maps $\wxuc, \wxud, \Fuw$ and $\Guw$.
	\begin{enumerate}[label = \roman*), leftmargin = 6mm]
		\item\label{item:wxd1} The set-valued maps $\wxuc$ and $\wxud$ are upper semicontinuous by a direct application of \cite[Lemma 5.15]{65}:
		the maps $\wxuc$ and $\wxud$ defined in \eqref{eq:wxu} have closed graphs because sets $\Cuw$ and $\Duw$ are closed, (to see this, note that $\gph \wxuc = \gph \wxud = S$)-- this leads to their outer semicontinuity by \cite[Lemma 5.10]{65}-- and by the assumption that $\wxuc$ and $\wxud$ are locally bounded;
		\item\label{item:wxd2} The maps $\wxuc$ and $\wxud$ have compact images:
		this property directly follows from outer semicontinuity and locally boundedness of $\wxuc$ and $\wxud$;
		\item The set-valued maps $\Fuw$ and $\Guw$ are upper semicontinuous by applying \cite[Lemma 5.15]{65} while noting that item \ref{item:A2p} and \ref{item:A3p} of Lemma~\ref{lem:hbc} hold;
		\item The maps $\Fuw$ and $\Guw$ have compact images, which follows from the fact that $\Fuw$ and $\Guw$ are locally bounded, and are outer semicontinuous.
	\end{enumerate}
	Moreover, continuously differentiability of $\Ly$ and the continuity of $\rc$ and $\rd$ imply the continuity of the functions been taken supremum in \eqref{eq:GammaC} and \eqref{eq:GammaD}.
	With the properties of $\wxuc, \wxud, \Fuw$ and $\Guw$., the single-valued maps $\Gamma_c$ and $\Gamma_d$ are upper semicontinuous by applying \cite[Proposition 2.9]{robust} twice while noting that for every $(x,u_c, w_c) \in (\dc) \setminus \wC\uw, \Gamma_c(x,u_c) = -\infty$ and for every $(x,u_d, w_d) \in (\dd) \setminus \wD\uw, \Gamma_d(x,u_d) = -\infty$.
	Then, applying Corollary~\ref{coro:2.13}, with $z = x, z' = u_c$ (or $z' = u_d$), $W = \uxc$ (or $W = \Ld$), and $w = \Gamma_c$ (or $w = \Gamma_d$, respectively) $\wS_c$ (or $\wS_d$, respectively) is lower semicontinuous.
	The maps $\wS_c$ and $\wS_d$ have nonempty values on $\Mc$ and $\Md$, respectively.
	This is because, first, $\uxc$ and $\Ld$ have nonempty values on $\Mc$ and $\Md$, respectively.
	In addition, since the inequalities in \eqref{eq:CLF-C} and \eqref{eq:CLF-D} hold, for each $(x, u_c) \in \Delta_c$, we have
	\IfConf{
		$\Gamma_c(x, u_c) + (1 - \sigma)\rc(x)\leq 0,$
		and for each $(x, u_d) \in \Delta_d$, we have
		$\Gamma_d(x, u_d) + (1 - \sigma)\rd(x) \leq 0.$
	}{
		$$\Gamma_c(x, u_c) + (1 - \sigma)\rc(x)\leq 0,$$
	and for each $(x, u_d) \in \Delta_d$,
	$$\Gamma_d(x, u_d) + (1 - \sigma)\rd(x) \leq 0.$$
}
	Then, since the functions $\rc$ and $\rd$ have positive values on $\Ir$ and $\Ls$, respectively, and $\sigma \in (0,1)$, for every $x \in \Mc$ (every $x \in \Md$), there exists $u_c \in \uxc(x)$ (exists $u_d \in \Ld(x)$) such that $\Gamma_c(x,u_c) < 0$ (respectively, $\Gamma_d(x,u_d) < 0$).
	Then, by the convexity of functions $\Gamma_c$ and $\Gamma_d$ in condition \ref{item:weakExist2} and of values of the set-valued maps $\uxc$ and $\Ld$ in \ref{item:weakExist1}, we have that the maps $\wS_c$ and $\wS_d$ have convex values on $\Mc$ and $\Md$, respectively.
	
	Then, to use \cite[Lemma 4.2]{75} for deriving regulation maps that are also lower semicontinuous, for each $x \in \reals^n$, we define the set-valued maps
	\IfConf{
	\begin{align}\label{eq:Sc}
		S_\star(x) : =
			\begin{cases}
				\wS_\star(x) & \ifeq x \in M_\star, \\
				\reals^{m_\star} & \otherw,
			\end{cases}
	\end{align}
with $\star \in \{c, d\}.$
}{
	\begin{align}\label{eq:Sc}
	S_c(x) : =
	\begin{cases}
		\wS_\star(x) & \ifeq x \in \Mc, \\
		\reals^{m_c} & \otherw,
	\end{cases}
\end{align}
	\begin{align}\label{eq:Sd}
		S_d(x) : =
		\begin{cases}
			\wS_d(x) & \ifeq x \in \Md,\\
			\reals^{m_d} & \otherw.
		\end{cases}
	\end{align}
}
	In addition, $S_c$ and $S_d$ also have nonempty and convex values due to the nonemptiness and convex-valued properties of $\wS_c$ and $\wS_d$.
	
	Now, according to Michael's Selection Theorem, namely, Theorem~\ref{thm:mselection}, there exist continuous functions $\widetilde{\kappa}_c: \reals^n \rightarrow \reals^{m_c}$ and $\widetilde{\kappa}_d: \reals^n \rightarrow \reals^{m_d}$ such that, for all $x \in \reals^n$,
	$$\widetilde{\kappa}_c(x) \in \overline{S_c(x)}, \qquad \widetilde{\kappa}_d(x) \in \overline{S_d(x)}.$$
	\IfConf{
	Now, with $\star \in \{c, d\}$, we define functions $\kappa_\star: \reals^n \rightarrow \reals^{m_\star}$ such that
	\vspace{-2mm}
		\begin{align}\label{eq:cselect}
			\kappa_\star(x) = \widetilde{\kappa}_\star(x) \in {\cal U}_\star  \qquad\forall x \in M_\star,
		\end{align}
	where the functions $\kappa_\star$ inherit the continuity of $\widetilde{\kappa}_\star$ on $M_\star$.
	}{
	Now, we define functions $\kc: \reals^n \rightarrow \reals^{m_c}$ and $\kd: \reals^n \rightarrow \reals^{m_d}$ such that
	\begin{align}\label{eq:cselect}
		\begin{split}
			&\kc(x) = \widetilde{\kappa}_c(x) \in \Uc  \qquad\forall x \in \Mc,\\
			&\kd(x) = \widetilde{\kappa}_d(x) \in \Ud  \qquad \forall x \in \Md,
		\end{split}
	\end{align}
	where the functions $\kc$ and $\kd$ inherit the continuity of $\widetilde{\kappa}_c$ and $\widetilde{\kappa}_d$ on $\Mc$ and $\Md$, respectively.
}
	Applying Lemma~\ref{lem:hbc}, the closed-loop system resulting from controlling $\widetilde{\HS}\uw$ by $\kc$ and $\kd$ in \eqref{eq:cselect} satisfies the hybrid basic conditions in Definition~\ref{def:hbc}.
	More precisely, this is because $\widetilde{\HS}\uw$ satisfies conditions (A1')-(A3') in Lemma~\ref{lem:hbc}, and the state-feedback pair $\kpair$ is continuous on $\Pc{\wC\uw}\cup \Pd{\wD\uw}$.
	With these properties and $\nabla \Ly$ being continuous, it follows that
	\begin{align*}
		&\kc(x) \in \uxc(x), \ \  \Gamma_c(x,\kc(x))\leq 0  &\forall x \in \Mc,\\
		&\kd(x) \in \Ld(x), \ \  \Gamma_d(x,\kd(x))\leq 0  &\forall x \in \Md,
	\end{align*}
	which lead to
	\IfConf{
		\begin{align}
			\begin{split}
				\sup_{\xi \in \Fuw(x,\kc(x), w_c)} &\inner{\nabla \Ly(x)}{\xi} + \sigma \rc(x) \leq  0\\
				& \hspace{1cm} \forall (x, \kc(x), w_c) \in \wC\uw,
			\end{split}\label{eq:CLF-ineqc}
		\end{align}
		\begin{align}
			\begin{split}
				\sup_{\xi \in \Guw(x,\kd(x), w_d)} \Ly(\xi) + \sigma \rd(x) - & \level \leq 0\\
				& \hspace{-12mm} \forall (x, \kd(x), w_d) \in \wD\uw.
			\end{split}\label{eq:CLF-ineqd}
		\end{align}
	\vspace{-2mm}
}{
	\begin{align}
		&\sup_{\xi \in \Fuw(x,\kc(x), w_c)} \inner{\nabla \Ly(x)}{\xi} + \rc(x) \leq 0 & \forall (x, \kc(x), w_c) \in \wC\uw,\label{eq:CLF-ineqc}\\
		&\sup_{\xi \in \Guw(x,\kd(x), w_d)}
			\Ly(\xi) + \rd(x) - \level \leq 0  & \forall (x, \kd(x), w_d) \in \wD\uw.\label{eq:CLF-ineqd}
	\end{align}
}

	The state feedback laws $\kc$ and $\kd$ can be extended -- not necessarily continuously -- to every point in $\Pc{\Cuw}$ and $\Pd{\Duw}$, respectively, by selecting values from the nonempty sets $\uxc(x)$ for every $x \in \Pc{\Cuw}$ and $\Ld(x)$ for every $x \in \Pd{\Duw}$.
	
	To complete the proof, we establish the robust controlled forward pre-invariance of $\Mr$.
	For this purpose, we apply Theorem~\ref{prop:Ly_pre} to the closed-loop system of $\Huw$ controlled via the extended state-feedback pair $\kpair$ that is defined on $\Pc{\Cuw} \cup \Pd{\Duw}$.
	Relationships \eqref{eq:CLF-ineqc} and \eqref{eq:CLF-ineqd} imply
	\IfConf{
		$\inner{\nabla \Ly(x)}{\xi} \leq 0$
		for every $(x, w_c) \in (\Ir \times \Wc) \cap \Ccl, \xi \in \Fcl(x, w_c)$, and
		$\Ly(\xi) \leq \level$
		for every $(x, w_d) \in (\Ls \times \Wd) \cap \Dcl, \xi \in \Gcl(x, w_d),$
		}{
	\begin{align*}
		&\inner{\nabla \Ly(x)}{\xi} \leq 0  &\forall &(x, w_c) \in (\Ir \times \Wc) \cap \Ccl, \xi \in \Fcl(x, w_c)\\
		&\Ly(\xi) \leq \level &\forall &(x, w_d) \in (\Ls \times \Wd) \cap \Dcl, \xi \in \Gcl(x, w_d),
	\end{align*}
}
	respectively.
	Thus, it is the case that \eqref{eq:ly1} and \eqref{eq:ly2} hold for the resulting closed-loop system.
	Moreover, since $\kd(x) \in \Ld(x)$ for every $x \in \Md$, \eqref{eq:rld} implies \eqref{eq:lyJump} for $\Hcl$.
	Hence, according to Definition~\ref{def:rcFI}, the extended state-feedback pair $\kpair$ renders the set $\Mr$ as in \eqref{eq:Mr} robustly controlled forward pre-invariant for $\Huw$.
\end{proof}

\ConfSp{-1mm}
\begin{remark}
	Item \ref{item:weakExist1} in Theorem~\ref{thm:weakExist} imposes lower semicontinuity of the mappings from state space to the input spaces at points where flows and jumps are allowed.
	For systems that does not have convex-valued $\uxc$ and $\Ld$ on $\Mc$ and $\Md$, respectively, Theorem~\ref{thm:weakExist} can still be applied if there exist nonempty, closed and convex subsets of $\uxc(x)$ and $\Ld(x)$ for every $x \in \Mc$ and $x \in \Md$, respectively, such that item \ref{item:weakExist2} holds for these subsets.
	Similar comments apply to the forthcoming results.
\end{remark}

To show existence of a state feedback pair $\kpair$ that renders $\K$ as in \eqref{eq:Mr} robustly forward invariant, we need further conditions on the regulation maps to ensure existence of a solution pair from every $\Pc{\Cuw}$.
Hence, we dedicate the remainder of this section to address, with a variation of RCLF for forward invariance in Definition~\ref{def:rclf}, the existence of a feedback pair for a class of $\Huw$ that induces robust controlled forward invariance of $\K$ by applying Theorem~\ref{prop:Ly}.
In particular, the next result resembles Theorem~\ref{thm:weakExist}, but employs different regulation maps to guarantee existence of nontrivial solution pairs and their completeness.
To this end, for every $x \in \Pc{\Cuw}$, we define the map
	\begin{align}\label{eq:rlc}
	\IfConf{
	\begin{split}
		&\Lc(x) : =\\
		&\hspace{2mm}\begin{cases}
			\{u_c \in \uxc(x): &\!\!\!\! \Fuw(x,u_d, 0) \cap T_{\Pc{\Cuw}}(x) \neq \emptyset\} \\
			&\forall x \in \partial \Pc{\Cuw} \setminus \Pd{\Duw}\\
			\uxc(x) & \otherw.
		\end{cases}
	\end{split}
		}{
		\Lc(x) : = \begin{cases}
			\{u_c \in \uxc(x): \Fuw(x,u_d, 0) \cap T_{\Pc{\Cuw}} \neq \emptyset\} & \forall x \in \partial \Pc{\Cuw} \setminus \Pd{\Duw}\\
			\uxc(x) & \otherw.
		\end{cases}
	}
	\end{align}
\begin{theorem}(existence of state-feedback pair for robust controlled forward invariance using RCLF for forward invariance) \label{thm:exist}
	Consider a hybrid system $\Huw = \datauw$ as in \eqref{eq:Huw} satisfying conditions \ref{item:A1p}-\ref{item:A3p} in Lemma~\ref{lem:hbc} and such that $\wxuc$ and $\wxud$ are locally bounded.
	Suppose there exists a pair $\clfpair$ that defines a robust control Lyapunov function for forward invariance of the sublevel sets of $\Ly$ for $\Huw$ as in Definition~\ref{def:rclf} with $\uxc$ in \eqref{eq:CLF-C} replaced by $\Lc$ as in \eqref{eq:rlc}.
	Let $\level < \rU$ satisfy \eqref{eq:CLF-rC}-\eqref{eq:CLF-D}, $\Ld$ be given as in \eqref{eq:rld}, and $\sigma \in (0,1).$
	If the following conditions hold:
	\begin{enumerate}[label = \ref{thm:exist}.\arabic*),leftmargin=12mm]
		\item\label{item:exist1} The set-valued maps $\Lc$ and $\Ld$ are lower semicontinuous, and $\Lc$ and $\Ld$ have nonempty, closed, and convex values on the set $\Pc{\Cuw}$ and the set $\Md$, respectively;
		\item\label{item:exist2} For each $x \in \Mc$, the function $u_c \mapsto \Gamma_c(x, u_c)$ in \eqref{eq:GammaC} is convex on $\Lc(x)$ and,
		for each $x \in \Md$, the function $u_d \mapsto \Gamma_d(x, u_d)$ in \eqref{eq:GammaD} is convex on $\Ld(x)$;
	\end{enumerate}
	then, the set $\Mr$ in \eqref{eq:Mr} is robustly controlled forward pre-invariant for $\Huw$ via a state-feedback pair $\kpair$ with $\kc$ being continuous on $\Mc$ and $\kd$ being continuous on $\Md$.
	Furthermore, if item~\ref{item:Ly4} in Theorem~\ref{prop:Ly} holds for the closed-loop system $\Hcl$ as in \eqref{eq:Hcl}, the pair $\kpair$ renders the set $\Mr$ robustly controlled forward invariant for $\Huw$.
\end{theorem}

\begin{proof}
	The robust forward pre-invariance of $\Mr$ for $\Huw$ follows from a direct application of Theorem~\ref{thm:weakExist}.
	More precisely, when conditions in Theorem~\ref{thm:exist} hold, every condition in Theorem~\ref{thm:weakExist} holds for a hybrid system $\wHS\uw$ that has flow map, jump map, and jump set given as $\Fuw$, $\Guw$, and $\Duw$, respectively, and flow set given by
	$$\widetilde{C}\uw = \{(x, u_c, w_c) \in \Cuw: u \in \Lc(x) \}.$$
	The set $\widetilde{C}\uw$ is closed.
	We show this by considering the sequence $(x_i, u_i, w_i) \in \widetilde{C}\uw$, for every $i$, converges to $(x,u,w)$, which is in $\Cuw$ since $\Cuw$ is closed.
	By definition of $\widetilde{C}\uw$, $u_i \in \Lc(x_i)$ for every $i$.
	Because $\Lc$ has closed values, $u \in \Lc(x)$.
	Hence, $(x, u, w) \in \widetilde{C}\uw$.
	Applying Theorem~\ref{thm:weakExist}, there exists a state-feedback pair $\kpair$ that renders $\Mr$ robustly controlled forward pre-invariant for $\wHS\uw$ with $\kc$ and $\kd$ being continuous on $\Mc$ and $\Md$, respectively.
	Since for every $x \in \Pc{\Cuw}$, such $\kc(x) \in \Lc(x) \subset \uxc(x)$, this implies such pair $\kpair$ is also $\Huw-$ admissible.
	Moreover, every solution pair to the closed-loop system resulting from $\Huw$ controlled by $\kpair$, i.e., $\Hw$, is also a solution pair to the closed-loop system of $\wHS\uw$ controlled by the same pair $\kpair$, i.e, $\wHS_w$.
	We show this via contradiction.
	Suppose there exist a solution pair $(\phi^*, w^*) \in \sol_{\Hw}$ such that $(\phi^*, w^*) \notin \sol_{\wHS_w}$.
	Since $\wHS\uw$ and $\Huw$ share the same jump map and jump set, if $\phi^*$ is pure discrete, then $(\phi^*, w^*)$ is also a solution pair to $\Huw$.
	In the case that $\phi^*$ is not pure discrete, by item~\ref{item:Sw1} of Definition~\ref{def:solution} and the fact that $\wHS\uw$ and $\Huw$ share the same flow map, there exists $j^\star$ with $I^{j^\star}$ with nonempty interior, such that
	\begin{align}
		&(\phi^*(t,j^\star), w^*(t,j^\star)) \in \Cw \label{eq:H1C} \\
		&(\phi^*(t,j^\star), w^*(t,j^\star)) \notin \wC_w.\label{eq:H2C}
	\end{align}
	Utilizing the projection maps introduced in Section~\ref{sec:HS} near \eqref{eq:wxu}, \eqref{eq:H1C} implies $\phi^*(t,j^\star) \in \Pwc{\Cw}$ and
	$$w^*(t,j^\star) \in \wxuc(\phi^*(t,j^\star), \kc(\phi^*(t,j^\star))).$$
	By definition of $\wC\uw$, $\Pwc{\wC_w} = \Pwc{\Cw}$, hence, together with \eqref{eq:H2C}, it must be that
	$$w^*(t,j^\star) \notin \{w_c : (\phi^*(t,j^\star), \kc(\phi^*(t,j^\star), w_c) \in \wC\uw\},$$
	which leads to the contradiction to the fact that $\wC\uw \subset \Cuw$.
	Hence, such $\kpair$ renders $\Mr$ robustly controlled forward pre-invariant for $\Huw$.

	According to Theorem~\ref{thm:4.1}, since the set $\Mc$ is closed, there exists a continuous extension of $\kc$ from $\Ir \cap \Pc{\Cuw}$ to $\reals^n$ with $\kc(x) \in \reals^m$ for every $x \in \inter \Ls \cap \Pc{\Cuw}$.\footnote{Note that the selected $\kc$ in proof of Theorem~\ref{thm:weakExist} is not necessarily continuous on $\Pc{\Cuw}$.}
	Then, applying such pair $\kpair$, with $\kc$ and $\kd$ being continuous on $\Ls \cap \Pc{\Cuw}$ and $\Md$, respectively, Lemma~\ref{lem:hbc} implies the closed-loop system is such that $\Fcl$ is outer semicontinuous, locally bounded and has nonempty and convex values on $(\Mr \times \Wc) \cap \Cw$.
	Hence, item~\ref{item:A2} in Definition~\ref{def:hbc} holds for closed-loop system $\wHS$.
	Then, applying Theorem~\ref{prop:Ly}, we show that the pair $\kpair$ renders set $\Mr$ robustly controlled forward invariant for $\wHS$.
	For every $x \in \Pwc{\Ccl},$ $(x, 0) \in \Cw$ by assumption.
	Inequalities \eqref{eq:ly1} and \eqref{eq:ly2} follow from \eqref{eq:CLF-ineqc} and \eqref{eq:CLF-ineqd} for the given pair $\clfpair$.
	Next, \eqref{eq:CLF-ineqc} implies condition \ref{item:Ly1}.
	Condition \ref{item:Ly2} follows from the definition of $\Lc$ in \eqref{eq:rlc}.
	Since \eqref{eq:CLF-ineqc} and the fact that $\rc(x)$ is positive for every $x \in M_c$, $\inner{\nabla V(x)}{\xi} < 0$, for every $x \in M_c$ and $\xi \in \Fuw(x, \kc(x), 0)$.
	Then, \eqref{eq:Ms} and \eqref{eq:rlc} together implies the feedback $\kc(x)$ selected from $\Lc(x)$ for every $x \in M_c$ are such that $\Fcl(x, 0) \cap T_{\Pwc{\Ccl}}(x) \neq \emptyset.$
	Thus, item \ref{item:Ly3} holds.
	Item \ref{item:Ly4} holds by assumption.
	The definition of $\Ld$ in \eqref{eq:rld} implies \eqref{eq:lyJump} holds.
	Hence, the set $\Mr$ is robustly controlled forward invariant for $\wHS$ via the selected $\kpair.$
	Furthermore, as showed above, the pair $\kpair$ is $\Huw-$ admissible and renders the set $\Mr$ robustly controlled forward invariant for $\Huw$ by Definition~\ref{def:rcFI}.
\end{proof}

Theorem~\ref{thm:exist} uses an alternative RCLF for forward invariance that is defined based on $\Lc$ as in \eqref{eq:rlc} instead of $\uxc$ as in Definition~\ref{def:rclf}.
This RCLF leads to the existence of state-feedbacks rendering $\K$ robust controlled forward invariance for $\Huw$.
By selecting $\kc$ from the map $\Lc$ in \eqref{eq:rlc} rather than the generic map $\uxc$, we guarantee existence of nontrivial solution pairs from every $x \in \Mr \setminus \Pd{\Duw}$.
This follows from an application of Lemma~\ref{lem:Tcone} and the fact that items \ref{item:Ly1}, \ref{item:Ly3}, and \ref{item:Ly4} in Theorem~\ref{prop:Ly} hold.
Moreover, item \ref{item:Ly4} ensures completeness of every $(\phi, w) \in \sol_{\Hcl}(\Mr)$.

\begin{remark}
	Results about selecting feedbacks from regulation maps for nominal hybrid systems (without perturbations), developed using a different set conditions and notion of control Lyapunov functions for forward invariance appeared in \cite{141}; see details in \cite[Definition 4.1]{141}.
	More precisely, the results in \cite{141} are derived from sufficient conditions for forward invariance of generic sets\NotConf{\footnote{The equivalent results of Corollary~\ref{coro:rcFI} in Section~\ref{sec:rcFI}.}} and are not tailored to sublevel sets of $\Ly.$
	In particular, in \cite{141}, to guarantee that the state component of every solution pair remains in $\Mr$, the feedback law $\kc$ needs to be locally Lipschitz, see \cite[Theorem 4.7, R4)]{141}.
	To get such a property, condition \cite[Theorem 4.7, R1')]{141} asks the regulation map $\wLc$ to be locally Lipschitz, leading to $\kc$ being a Lipschitz selection.
	By exploiting results in Section~\ref{sec:rcFI-Lya}, Theorem~\ref{thm:exist} only requires $\kc$ to be a continuous selection.
\end{remark}

\begin{remark}\label{rem:relax}
	In the case where control inputs affect only the jumps, the conditions in Theorem~\ref{thm:weakExist} lead to robustly controlled forward invariance of $\Huw$, provided \eqref{eq:ly1} holds during flows.
	Similarly, when control inputs affect only the flows, the conditions involving $\Fuw$ and $\Cuw$ in Theorem~\ref{thm:weakExist}, together with \eqref{eq:ly2}, lead to robust controlled forward invariance of $\Mr$.
	In addition, the results in this section can be applied to purely continuous-time and purely discrete-time systems by defining RCLF for forward invariance only based on \eqref{eq:CLF-C} or \eqref{eq:CLF-D}, respectively.
\end{remark}

\begin{example}(Existence of continuous state-feedback control law for the bouncing ball)\label{ex:bball-exists}
	First, since 
	$$\Md = \{0\} \times [-\sqrt{2\Emax}, -\sqrt{\gamma \hmin}]$$
	and $\uxd(x) = \Ud$, \ref{item:weakExist1} in Theorem~\ref{thm:weakExist} holds for $\Huw$.
	Following the steps in Section~\ref{sec:design}, we construct the regulation map $\Gamma_d.$
	Since there is no control input during flows, we omit defining $\Gamma_c$.
	Moreover, since the input $u_d$ is only active when $(x, u_d, w_d) \in D_1$, we define the map $\Gamma_d$ based on $G_1$ only.
	Then, for $\level = -\gamma\hmin$ and for every $(x, u_d) \in \{(x, u_d) \in \reals^2 \times \Ud: (x, u_d, w_d) \in (\Ls \times \Ud \times \Wd) \cap D_1\}$, with $\sigma = \frac{1}{2}$, $\Gamma_d$ is given by
	\begin{align*}
		\Gamma_d(x, u_d)
		&= \max\limits_{w_d \in [e_1, e_2]} \Ly(G_1(x, u_d, w_d)) + \frac{\rd(x)}{2} - \level\\
		&= -\frac{(u_d - e_1x_2)^2}{2} + \gamma \left(\frac{\varepsilon}{2} + \hmin\right).
	\end{align*}
	Item \ref{item:weakExist2} in Theorem~\ref{thm:weakExist} holds since, for each $x \in \Md$, the function $u_d \mapsto \Gamma_d(x, u_d)$ is convex on $\Ld(x)$.
	For each $x \in \reals^2$, the map $S_d$ in \eqref{eq:Sc} is given by
	\begin{align}\label{eq:BSd}
		\hspace{-5mm}
		S_d(x) : =
		\begin{cases}
			\{u_d \in \Ld(x): &\!\!\!\! \gamma (\frac{\varepsilon}{2} + \hmin) -\frac{(u_d - e_1x_2)^2}{2} < 0\} \\
			&\ifeq x \in \Ls \cap \Pd{D_1}, \\
			\reals & \otherw.
		\end{cases}
	\end{align}
	In addition, $\Huw$ given in \eqref{eq:Bplant} satisfies conditions \ref{item:A1p} - \ref{item:A3p} in Lemma~\ref{lem:hbc}.
	According to Theorem~\ref{thm:weakExist}, there exists a state feedback $\kd : \reals^2 \rightarrow \reals$ that is continuous on $\Md$.
	In particular, such a feedback is selected from the closure of the map $S_d$ given in \eqref{eq:BSd}, which reduces to an interval:
	\IfConf{
		\begin{align}\label{eq:BSdcl}
			\overline{S_d}(x): = \left[\max\left\{\sqrt{2\gamma\left(\frac{\varepsilon}{2} + \hmin\right) } + e_1x_2,0\right\},\right.\\
			\left. \sqrt{2\Emax} + e_2x_2\right].
		\end{align}
	}{
		\begin{align}\label{eq:BSdcl}
			\overline{S_d}(x): = \left[\max\left\{\sqrt{2\gamma\left(\frac{\varepsilon}{2} + \hmin\right) } + e_1x_2,0\right\}, \sqrt{2\Emax} + e_2x_2\right].
		\end{align}
	}
	
	One such continuous selection is 
	\begin{align}\label{eq:Bkd}
		\kd(x) : = \sqrt{\frac{\gamma(\frac{\varepsilon}{2} + \hmin)}{\Emax}}x_2 + \sqrt{2\gamma\left(\frac{\varepsilon}{2} + \hmin\right)}.
	\end{align}
	
	\IfConf{
	Applying \cite[Theorem~4.15 and Lemma~4.12]{PartI-Arxiv}, we verify that our design of $\kd$ in \eqref{eq:Bkd} indeed renders $\Mr$ robustly controlled forward invariant for $\Huw$.
	To this end, we check the corner cases of jumps  from $\K \cap \Pd{D_1}$ and from $\K \cap \Pd{D_2}$.
}{
	Since Corollary~\ref{coro:rcFI} provides conditions guaranteeing robust controlled forward invariance for hybrid systems without a Lyapunov function, we verify that our design of $\kd$ in \eqref{eq:Bkd} indeed renders $\Mr$ robustly controlled forward invariant for $\Huw.$
	To this end, first, $\K$ is a subset of $C \cup \Pd{\Duw}$, $F$ is Lipschitz and $F(x)$ is convex on $C$ by construction and \ref{item:rcFI4} holds since $\K \cap C$ is compact.
	Then, item \ref{item:rcFI1} and \ref{item:rcFI5} hold true trivially; while item \ref{item:rcFI3} holds since \eqref{eq:BVdot} and item 1) of Lemma~\ref{lem:innerTcone}.
	Finally, for the closed-loop system with $u_d$ replaced by $\kd$ in \eqref{eq:Bkd}, we check the extreme cases for every $x \in \K \cap \Pd{D_1}$ and every $x \in \K \cap \Pd{D_2}$.
}
	More precisely, the worst case for impact with zero height is when $x$ is such that $x_2 = - \sqrt{2\gamma\hmin}$ before the impact and, after the impact, $x$ is updated by the map $G_1(x, \kd(x), e_1)$, i.e.,
	\IfConf{
	$G_1(x, \kd(x), e_1) \geq \sqrt{2\gamma\left(\frac{\varepsilon}{2} + \hmin\right)}$
	since $\sqrt{\gamma\left(\frac{\varepsilon}{2} + \hmin\right)} < e_1 \sqrt{\Emax}$.
	}{
		\begin{align}
			G_1(x, \kd(x), e_1)
			&= \sqrt{\frac{\gamma\left(\frac{\varepsilon}{2} + \hmin\right)}{\Emax}}x_2 + \sqrt{2\gamma\left(\frac{\varepsilon}{2} + \hmin\right)} - e_1x_2\\
			&= \sqrt{2\gamma\left(\frac{\varepsilon}{2} + \hmin\right)} + \left(\sqrt{\frac{\gamma\left(\frac{\varepsilon}{2} + \hmin\right)}{\Emax}} - e_1\right)(- \sqrt{2\gamma\hmin}),
		\end{align}
	which is greater than $\sqrt{2\gamma\left(\frac{\varepsilon}{2} + \hmin\right)}$ since $\sqrt{\gamma\left(\frac{\varepsilon}{2} + \hmin\right)} < e_1 \sqrt{\Emax}$.
	Then, \ref{item:rcFI2} holds for every $x \in \K \cap \Pd{D_2}$ since \eqref{eq:Bg2}.}
	
	Simulations are generated to show solutions to $\Huw$ controlled by $\kd$ in \eqref{eq:Bkd} with system parameters
	$\gamma = 9.81, \hmin = 10, \hmax = 12, \vmax = 6\sqrt{\gamma}, e_1 = 0.8, e_2 = 0.9, e_p = 0.95, \varepsilon = 0.1$, and $\delta_p = 0.01.$\footnote{All simulations in this section are generated via the Hybrid Equations (HyEQ) Toolbox for MATLAB; see \cite{74}.
		Code available at https://github.com/HybridSystemsLab/InvariantBoucingBall
	and at https://github.com/HybridSystemsLab/InvariantPointMass}
	Over the simulation horizon, the disturbance $w_d$ is randomly generated within interval $[e_1, e_2]$, and updated after each impact.
	One solution that starts from the initial condition for $x(0,0) = (11, 0)$ is shown in Figure~\ref{fig:bball1}.
	Figure~\ref{fig:bball1}(a) presents the randomly generated disturbance $w_d$ for $\Huw$.
	Moreover, even under the effect of the disturbance, as desired, the peaks of the resulting height reach values larger than $\hmin$ and smaller than $\hmax$ as Figure~\ref{fig:bball1}(a) shows.
	Figure~\ref{fig:bball1}(b) shows, on the $(x_1, x_2)$ plane, that the solution stays within the set $\K$ for all time, which is the region bounded by dark green dashed line.
	\hfill $\triangle$
\end{example}

\ConfSp{-3mm}
		\setlength{\unitlength}{0.32\textwidth}
	\begin{figure}[h]
		\scriptsize
		\begin{subfigure}[t]{0.5\textwidth}
			\centering
			\includegraphics[trim = {0 4.5cm 0 0},clip,width = \unitlength]{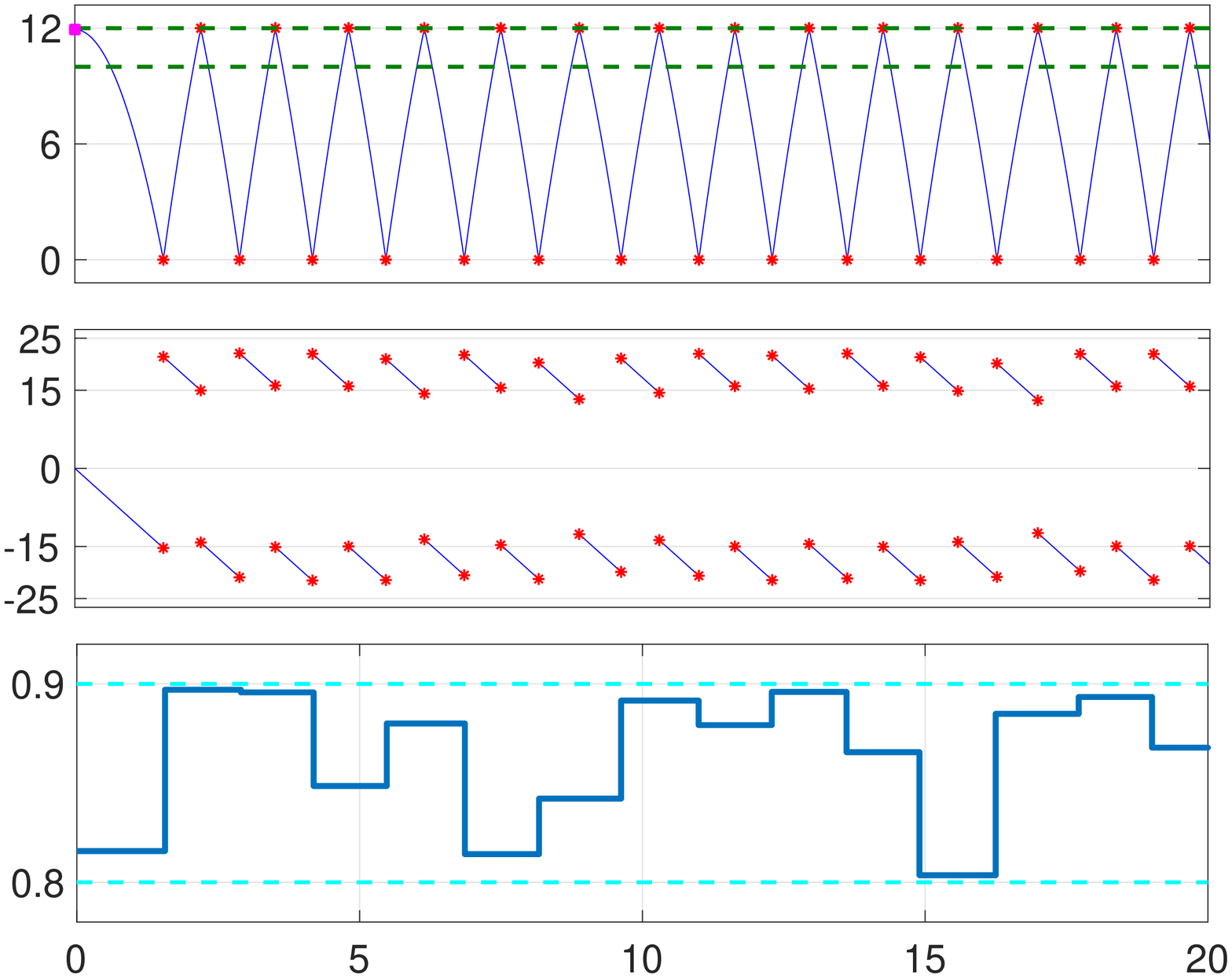}
			\put(-.02, .78){$\hmax$}
			\put(-.02, .74){$\hmin$}
			\put(-1.03, .7){$x_1$}
			\put(-1.03, .43){$x_2$}
			\put(-0.02, .26){$e_2$}
			\put(-0.02, .1){$e_1$}
			\put(-1.03, .18){$w_d$}
			\put(-.54, -.01){$t$ [sec]}
			\caption{Height and velocity of the ball and $w_d$.}
			\label{fig:sig1}
		\end{subfigure}
		\begin{subfigure}[t]{0.5\textwidth}
			\centering
			\includegraphics[trim = {0 4.5cm 0 0},clip,width = \unitlength]{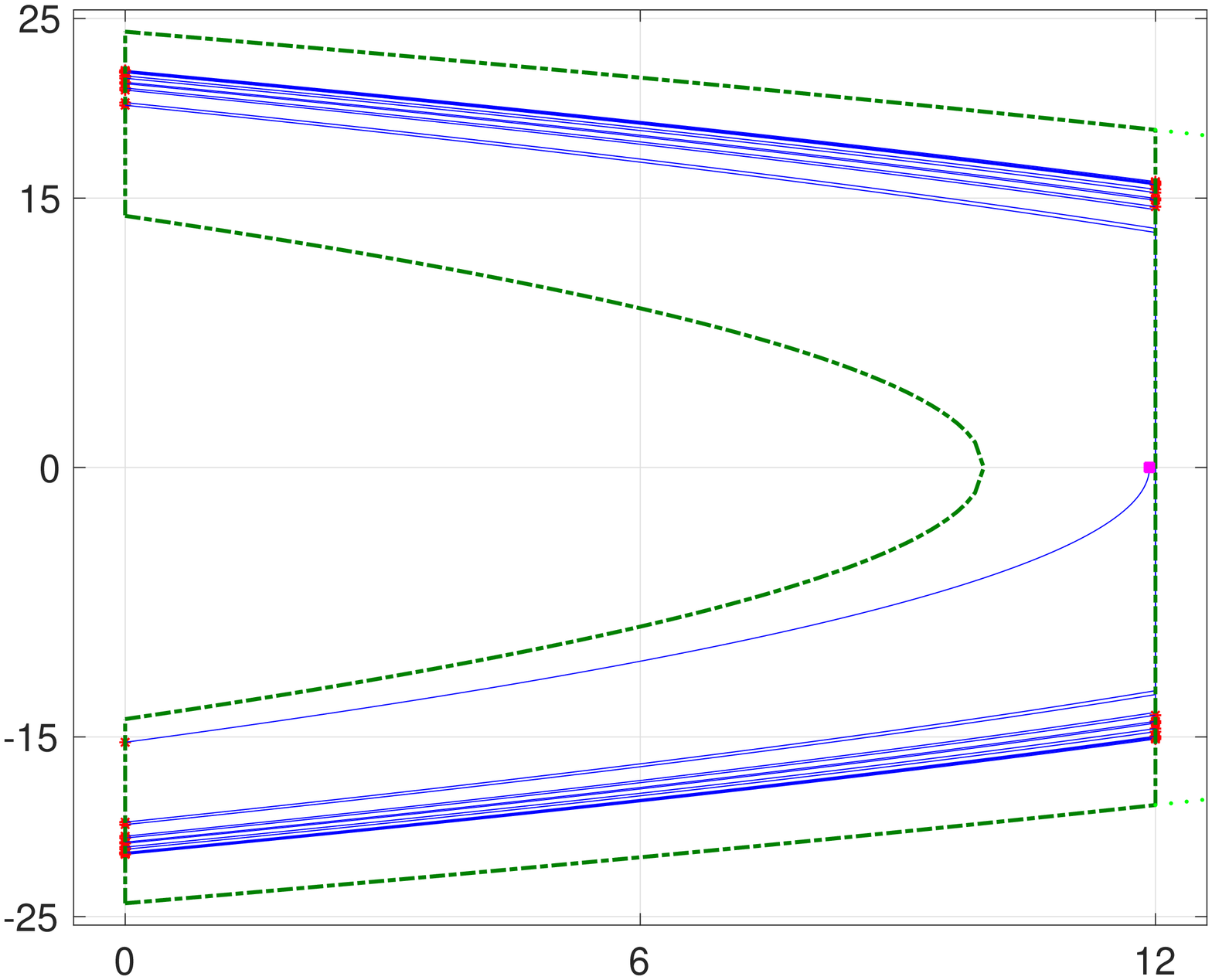}
			\put(-1.06,.43){$x_2$}
			\put(-.5,-.03){$x_1$}
			\put(-.2,.3){$\K$}
			\put(-.18,.46){(11,0)}
			\caption{State component of closed-loop solution on the $(x_1, x_2)$ plane.}
			\label{fig:tball1}
		\end{subfigure}
		\caption{\label{fig:bball1} Simulation of $\Huw$ controlled by $\kd$ in \eqref{eq:Bkd}.}
	\end{figure}
\ConfSp{-3mm}

	In the next example, we apply results in this section to design an invariance-based controller for the robotic manipulator introduced in Example~\ref{ex:hmass-rclf}.
	\begin{example}(Existence of continuous feedback control law for the robotic manipulator)
		Consider the system $\Huw$ in Example~\ref{ex:hmass-rclf}.
		For this system, the set $\Mc$ in \eqref{eq:Ms} is equal to $\Ir$.
		Furthermore, since $\partial \Pc{\Cu} \setminus D = \emptyset$, for every $x \in \Mc$, we have $\Lc(x) = \Psi_c(x)$.
		Thus, item~\ref{item:exist1} in Theorem~\ref{thm:exist} holds.
		Next, we construct $\Gamma_c$ and the regulation map following the steps in Section~\ref{sec:rclf-FI}.\footnote{Due to the absence of control inputs during jumps, we omit defining $\Gamma_d$.}
		For $\level < \rU = b\overline{v}^2$ and for every $(x, u_c) \in \{(x, u_c) \in \reals^2 \times \Uc: x \in \Ls\}$, with $\sigma = \frac{1}{2}$, $\Gamma_c$ is given by
		\IfConf{
			$\Gamma_c(x, u_c) = \max\limits_{\xi \in F_u(x,u_c)} \inner{\nabla \Ly(x)}{\xi} - \frac{1}{2} \rc(x).$
		}{
		\begin{align}\label{eq:raGc}
			\Gamma_c(x, u_c) = \max\limits_{\xi \in F_u(x,u_c)} \inner{\nabla \Ly(x)}{\xi} - \frac{1}{2} \rc(x).
		\end{align}
	}
		As presented in Example~\ref{ex:hmass-rclf}, when \eqref{eq:detQ} and \eqref{eq:trQ} hold, the continuous feedback law
		\begin{align}\label{eq:hmass-law}
			\kc(x) = -k_p x_1 - k_d x_2
		\end{align}
		renders the set $\K$ in \eqref{eq:raK} robust controlled forward invariant for $\Huw$ therein.
		The existence of such continuous feedback follows from Theorem~\ref{thm:weakExist} since, for each $x \in M_c$, $u_c \mapsto \Gamma_c(x,u_c)$ is convex on $\Lc(x)$ and $\Huw$ satisfies conditions \ref{item:A1p} - \ref{item:A3p} in Lemma~\ref{lem:hbc}.
		Next, we design the gain of such a feedback law to satisfy \eqref{eq:detQ} and \eqref{eq:trQ}.
		Consider $\level = \frac{4}{5}\rU = \frac{4}{5}b\overline{v}^2$ and the RCLF, i.e., $\Ly$ in \eqref{eq:K-arm}, that is defined with
		$P = \begin{bmatrix}5 &1\\ 1 & 2\end{bmatrix}$.
		The working environment has parameters $k_c = 0.1$ and $b_c = 0.02$, the velocity threshold is $\overline{v} = 0.6$, the coefficient of restitution parameters are $e_2 = 0.9$ and $e_1 = 0.8$, and the maximum allowed input is $f_{max} = 10$.
		We simulate several solutions to $\Huw$ controlled by $\kc$ given in \eqref{eq:hmass-law} with gain $k = [-0.5 \ -2].$
		\ConfSp{-2mm}
		\setlength{\unitlength}{0.38\textwidth}
		\begin{figure}[H]
			\scriptsize
			\centering
			\includegraphics[width = \unitlength]{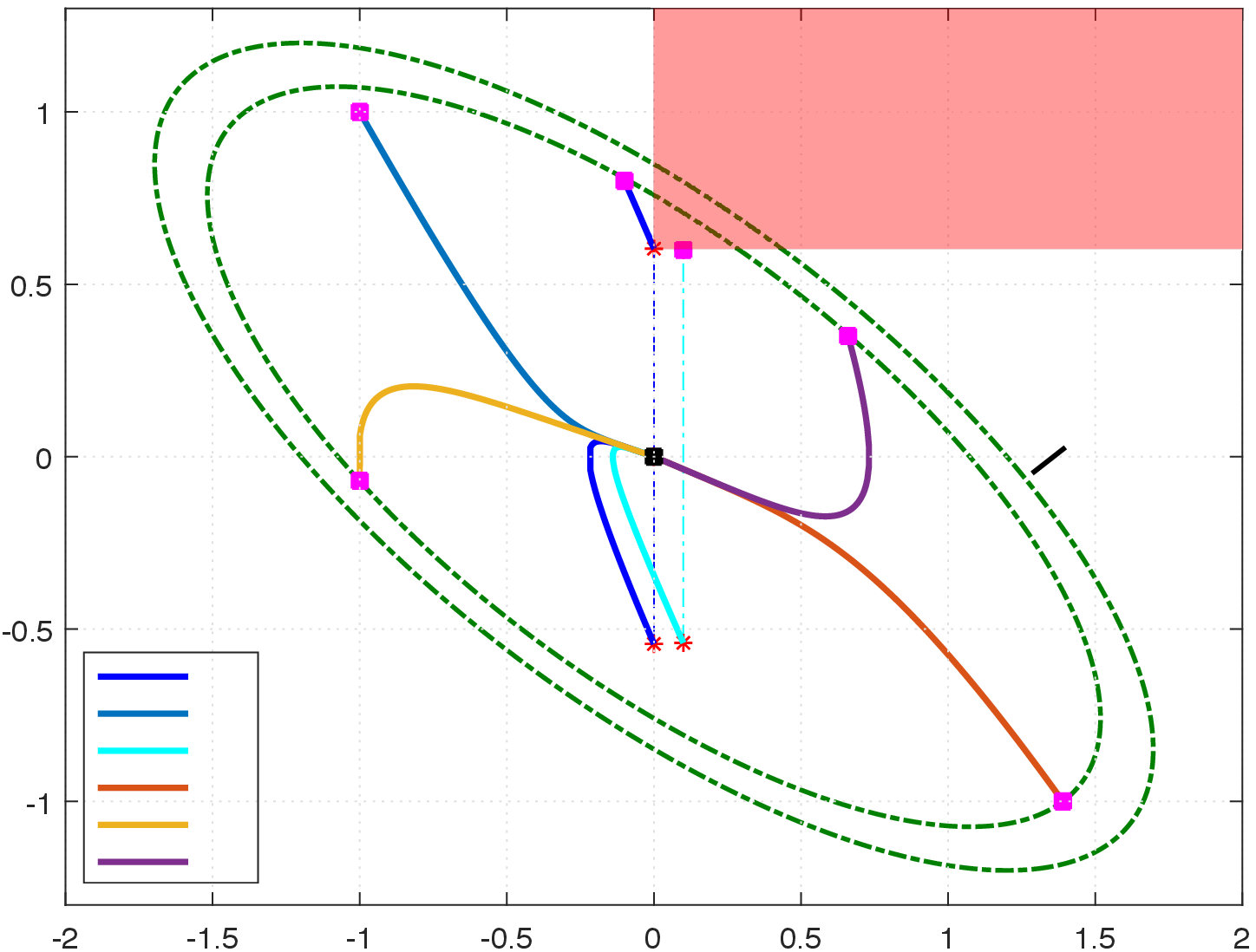}
			\put(-0.84, .215){$s_1$}
			\put(-0.84, .185){$s_2$}
			\put(-0.84, .155){$s_3$}
			\put(-0.84, .125){$s_4$}
			\put(-0.84, .095){$s_5$}
			\put(-0.84, .065){$s_6$}
			\put(-0.4, .16){$\K$}
			\put(-0.14, .62){$\Pi(D_w)$}
			\put(-0.18, .45){$\Ly^{-1}(\rU)$}
			\put(-1.03, .4){$x_2$}
			\put(-.5, -.02){$x_1$}
			\caption{Simulation of $\Huw$ in Example~\ref{ex:hmass-model} controlled by $\kc$ in \eqref{eq:hmass-law}.}
			\label{fig:hmass}
		\end{figure}
		\ConfSp{-2mm}
		\noindent As shown in Figure~\ref{fig:hmass}, the inner green dash line is the boundary of set $\K$ and the outer green dash line is the $\rU$-level set of $\Ly$.
		Six solutions are shown in Figure~\ref{fig:hmass}.
		Each solution starts with an initial condition (labeled as square pink points) that is within the set $\K$ and converges to the origin (labeled as a square black point) in the limit.
		Solutions labeled $s_1$ and $s_3$ exhibit jumps when the trajectory reach set $D$ (the shaded red square), and the jumps are represented with red stars and dotted lines that match the color of each solution.
		Note that all solutions stay within the set $\K$, as expected.
		\hfill $\triangle$
	\end{example}

%% file: Sections/6-Min-Norm.tex
Inspired by the pointwise minimum norm results in \cite{robust} and \cite[Theorem 5.1]{135}, we construct state-feedback pairs rendering the set $\Mr$ as in \eqref{eq:Mr} robust controlled forward invariant.
We employ Theorem~\ref{thm:weakExist} to show that the resulting closed-loop has the desired property.

For a given pair $\clfpair$ defining a RCLF for forward invariance as in Definition~\ref{def:rclf}, we first construct appropriate functions $\Gamma_c, \Gamma_d$ and regulation maps $S_c, S_d$ in Section~\ref{sec:design}.
When \ref{item:weakExist2} in Theorem~\ref{thm:weakExist} holds, $u_c \mapsto \Gamma_c(x, u_c)$ is convex on $\uxc(x)$ for every $x \in \Mc$, and $u_d \mapsto \Gamma_d(x, u_d)$ is convex on $\Ld(x)$ for every $x \in \Md$.
Hence, the maps $S_c$ and $S_d$ have nonempty and convex values on $\reals^n$.
According to \cite[Theorem 4.10]{rudin1987real}, for every $x \in \Lss \cap \Pc{\Cuw}$ and $x \in \Lss \cap \Pd{\Duw}$, respectively, the closure of $S_c(x)$ and $S_d(x)$, i.e., $\overline{S_c(x)}$ and $\overline{S_d(x)}$, have unique element of minimum norm.
Thus, we construct the state-feedback laws $\kmc : \Lss \cap \Pc{\Cuw} \rightarrow \Uc$ and $\kmd : \Lss \cap \Pd{\Duw} \rightarrow \Ud$ as
\begin{align}\label{eq:minSelect}
	\begin{split}
			\kmc(x) := \argmin\limits_{u_c \in \overline{S_c(x)}} |u_c| \qquad \forall &x \in \Lss \cap \Pc{\Cuw},\\
			\kmd(x) := \argmin\limits_{u_d \in \overline{S_d(x)}} |u_d| \qquad \forall &x \in \Lss \cap \Pd{\Duw}.
	\end{split}
\end{align}
Moreover, such state-feedback pair enjoys continuity when the maps $\uxc$ and $\Ld$ satisfy \ref{item:weakExist1}.
We capture these in the following result.
\begin{theorem}(pointwise minimum norm state-feedback laws for robust controlled forward pre-invariance) \label{thm:premin}
	Consider a hybrid system $\Huw$ as in \eqref{eq:Huw} satisfying conditions \ref{item:A1p}-\ref{item:A3p} in Lemma~\ref{lem:hbc}.
	Suppose there exists a pair $\clfpair$ that defines a robust control Lyapunov function for forward invariance of $\Huw$ as in Definition~\ref{def:rclf}.
	Let $\level < \rU$ satisfy \eqref{eq:CLF-rC}-\eqref{eq:CLF-D} and $\Ld$ be given as in \eqref{eq:rld}.
	Furthermore, suppose conditions \ref{item:weakExist1} and \ref{item:weakExist2} in Theorem~\ref{thm:weakExist} hold.
	Then, the state-feedback pair $(\kmc, \kmd)$ given as in \eqref{eq:minSelect} renders the set $\Mr$ in \eqref{eq:Mr} robustly controlled forward pre-invariant for $\Huw$.
	Moreover, $\kmc$ and $\kmd$ are continuous on set $\Mc$ and $\Md$ as in \eqref{eq:Ms}, respectively.
\end{theorem}

\begin{proof}
	The first claim follows from similar proof steps in Theorem~\ref{thm:weakExist}.
	In particular, since $\kmc$ and $\kmd$ are selected from the closure of $S_c$ and $S_d$, i.e.,
	\IfConf{
	$\kmc(x) \in \overline{S_c(x)}$ and $\kmd(x) \in \overline{S_d(x)},$
}{
	$$\kmc(x) \in \overline{S_c(x)}, \quad \text{ and } \quad \kmd(x) \in \overline{S_d(x)},$$}
	it follows that
	\begin{align*}
		&\kmc(x) \in \uxc(x), \qquad \Gamma_c(x,\kmc(x))\leq 0  \qquad \forall x \in \Mc,\\
		&\kmd(x) \in \Ld(x), \qquad \Gamma_d(x,\kmd(x))\leq 0  \qquad \forall x \in \Md,
	\end{align*}
	which lead to
	\IfConf{
	\begin{align}
		\begin{split}
		\sup_{\xi \in \Fuw(x,\kmc(x), w_c)} \inner{\nabla \Ly(x)}{\xi} &+ \rc(x) \leq 0\\
		&\forall (x, \kmc(x), w_c) \in \wC\uw,\\
		\sup_{\xi \in \Guw(x,\kmd(x), w_d)} \Ly(\xi) + \rd&(x) - \level \leq 0\\
		&\forall (x, \kmd(x), w_d) \in \wD\uw.
		\end{split}\label{eq:minV}
	\end{align}
	}{
	\begin{align}
		\begin{split}
			&\sup_{\xi \in \Fuw(x,\kmc(x), w_c)} \inner{\nabla \Ly(x)}{\xi} + \rc(x) \leq 0 \qquad \forall (x, \kmc(x), w_c) \in \wC\uw,\\
			&\sup_{\xi \in \Guw(x,\kmd(x), w_d)} \Ly(\xi) + \rd(x) - \level \leq 0  \qquad \ \ \forall (x, \kmd(x), w_d) \in \wD\uw.
		\end{split}\label{eq:minV}
	\end{align}
}
	The feedback pair $(\kmc, \kmd)$ can be extended to every point in $\Pc{\Cuw}$ and $\Pd{\Duw}$, respectively, by selecting values from the nonempty sets $\uxc(x)$ for every $x \in \Pc{\Cuw}$ and $\Ld(x)$ for every $x \in \Pd{\Duw}$.
	Then, applying Theorem~\ref{prop:Ly_pre}, we establish the robust controlled forward pre-invariance of $\Mr$ for $\Huw$ via $(\kmc, \kmd)$.
	
	Finally, the continuity of $\kmc$ and $\kmd$ follow directly from Proposition~\ref{prop:2.19}.
	In particular, maps $\overline{S_c}$ and $\overline{S_d}$ are lower semicontinuous with nonempty closed convex values as shown in proof of Theorem~\ref{thm:weakExist}.
\end{proof}
\ConfSp{-2mm}
A similar result to Theorem~\ref{thm:premin} can be derived using Theorem~\ref{thm:exist} to render $\Mr$ robustly controlled forward invariant for $\Huw$ via $(\kmc, \kmd)$.
In such a case, the feedback law $\kmc$ is selected from the closure of a map $S_c$ that is defined using $\Lc$ given as in \eqref{eq:rlc} instead of using $\uxc.$
More precisely, we consider the state feedback laws $\kmc$ defined as in \eqref{eq:minSelect} with $S_c$ given by
\ConfSp{-2mm}
\begin{align}\label{eq:Scmin}
	\hspace{-2mm} S_c(x) : =
	\begin{cases}
		\{u_c\in \Lc(x):\Gamma_c(x, u_c) < 0 \} & \ifeq x \in \Mc, \\
		\reals^{m_c} & \otherw.
	\end{cases}
\end{align}
In addition to conditions \ref{item:exist1} and \ref{item:exist2} in Theorem~\ref{thm:exist}, robustly controlled forward invariance of $\Mr$ requires item \ref{item:Ly4} in Theorem~\ref{prop:Ly} to hold for the closed-loop system $\Hcl$.
We formally present such a result as follows.
\begin{theorem}(pointwise minimum norm state-feedback laws for robust controlled forward invariance) \label{thm:min}
	Consider a hybrid system $\Huw$ as in \eqref{eq:Huw} satisfying conditions \ref{item:A1p}-\ref{item:A3p} in Lemma~\ref{lem:hbc}.
	Suppose there exists a pair $\clfpair$ that defines a robust control Lyapunov function for forward invariance for $\Huw$ as in Definition~\ref{def:rclf}.
	Let $\level < \rU$ satisfy \eqref{eq:CLF-rC}-\eqref{eq:CLF-D}, $\Lc$ and $\Ld$ be given as in \eqref{eq:rlc} and \eqref{eq:rld}, respectively.
	Furthermore, suppose conditions \ref{item:exist1} and \ref{item:exist2} in Theorem~\ref{thm:exist} hold.
	Then, the state-feedback pair $(\kmc, \kmd)$ given as in \eqref{eq:minSelect} defined using $S_c$ as in \eqref{eq:Scmin} renders the set $\Mr$ in \eqref{eq:Mr} robustly controlled forward invariant for $\Huw$ if condition \ref{item:Ly4} in Theorem~\ref{prop:Ly} holds for the closed-loop system $\Hcl$.
	Moreover, $\kmc$ and $\kmd$ are continuous on the sets $\Mc$ and $\Md$ as in \eqref{eq:Ms}, respectively.
\end{theorem}

\begin{proof}
	The proof resembles the one for Theorem~\ref{thm:premin}.
	In particular, the selection $(\kmc, \kmd)$ given as in \eqref{eq:minSelect} defined using $S_c$ as in \eqref{eq:Scmin} leads to
	\begin{align*}
		&\kmc(x) \in \Lc(x), \qquad \Gamma_c(x,\kmc(x))\leq 0  \qquad \forall x \in \Mc,\\
		&\kmd(x) \in \Ld(x), \qquad \Gamma_d(x,\kmd(x))\leq 0  \qquad \forall x \in \Md,
	\end{align*}
	which, in turn, leads to the inequalities in \eqref{eq:minV}.
	The feedback pair $(\kmc, \kmd)$ can be extended to every point in $\Pc{\Cuw}$ and $\Pd{\Duw}$, respectively, by selecting values from the nonempty sets $\Lc(x)$ for every $x \in \Pc{\Cuw}$ and $\Ld(x)$ for every $x \in \Pd{\Duw}$.
	Then, applying Theorem~\ref{prop:Ly}, we establish robust controlled forward pre-invariance of $\Mr$ for $\Huw$ via $(\kmc, \kmd)$ with the addition of condition \ref{item:Ly4} in Theorem~\ref{prop:Ly} for the closed-loop system $\Hcl$.
	Then, the continuity of $\kmc$ and $\kmd$ follow directly from Proposition~\ref{prop:2.19}.
\end{proof}

\ConfSp{-3mm}
	Next, applying Theorem~\ref{thm:min}, a control law with minimum point-wise norm rendering the set $\K$ in \eqref{eq:BK} robustly controlled forward invariant for the bouncing ball system $\Huw$ is provided.

\begin{example}(Minimum norm selection for the bouncing ball system)
	Consider the feedback law
	$$\kmd (x) = \argmin\limits_{u_d \in \overline{S_d(x)}} |u_d|,$$
	where $\overline{S_d(x)}$ is as in \eqref{eq:BSdcl}.
	It leads to the continuous state-feedback law
	\begin{align}\label{eq:Brd}
		\kmd(x) = \max\left\{\sqrt{2\gamma \left(\frac{\varepsilon}{2} + \hmin\right)} + e_1x_2, 0\right\},
	\end{align}
	for every $x \in \K \cap \Pd{D_1}$.
	Following same steps as in Example~\ref{ex:bball-exists}, it can be shown that $\K$ in \eqref{eq:BK} is robustly controlled forward invariant for $\Huw$ via $\kmd$.
	
	Simulations are generated for $\Huw$ controlled by $\kmd$ given as in \eqref{eq:Brd} with the same system settings as in Example~\ref{ex:bball-exists}.
	One solution that starts from the same initial condition $x = (11,0)$ is shown in Figure~\ref{fig:bball0}.
	As shown in Figure~\ref{fig:bball0}(a), the peaks of the height in between impacts are between $\hmin = 10$ and $\hmax = 12$, while on the $(x_1, x_2)$ plane, the trajectory stays within the set $\K$, which is the region bounded by dark green dashed lines.
	
	As expected, compared to Figure~\ref{fig:bball1}(a), we observe in Figure~\ref{fig:bball0}(a) that there are only 7 impacts with the controlled surface within the time span of 0 to 20 seconds ; while there are 14 impacts in Figure~\ref{fig:bball1}(a) and every impact is followed with a pull.
	This indicates that less energy is used to bounce the ball at the controlled surface to maintain peak position within range $[\hmin, \hmax]$.
	This is also verified by the input values from both controllers, where the state-feedback $\kmd$ has smaller value than the controller $\kd$ in Example~\ref{ex:bball-exists}.
	\hfill $\triangle$
\end{example}

\ConfSp{-3mm}
\setlength{\unitlength}{0.32\textwidth}
\begin{figure}[H]
	\scriptsize
	\begin{subfigure}[t]{0.5\textwidth}
		\centering
		\includegraphics[trim = {0 4cm 0 0},clip,width = \unitlength]{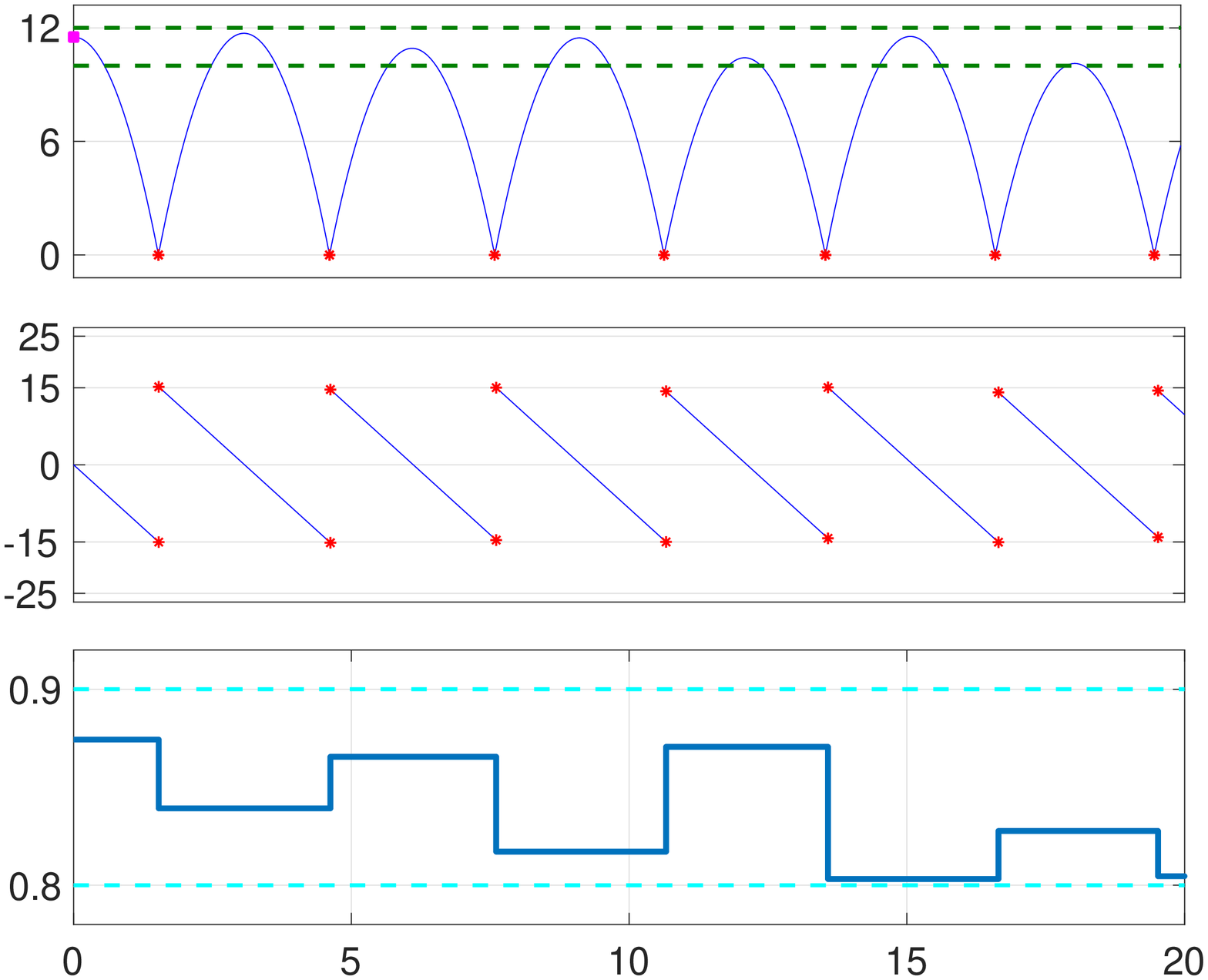}
		\put(-.02, .78){$\hmax$}
		\put(-.02, .74){$\hmin$}
		\put(-1.03, .7){$x_1$}
		\put(-1.03, .43){$x_2$}
		\put(-0.02, .26){$e_2$}
		\put(-0.02, .1){$e_1$}
		\put(-1.03, .18){$w_d$}
		\put(-.54, -.01){$t$ [sec]}
		\caption{Ball position and velocity and $w_d$.}
		\label{fig:sig0}
	\end{subfigure}
	\begin{subfigure}[t]{0.5\textwidth}
		\centering
		\includegraphics[trim = {0 4.2cm 0 0},clip,width = \unitlength]{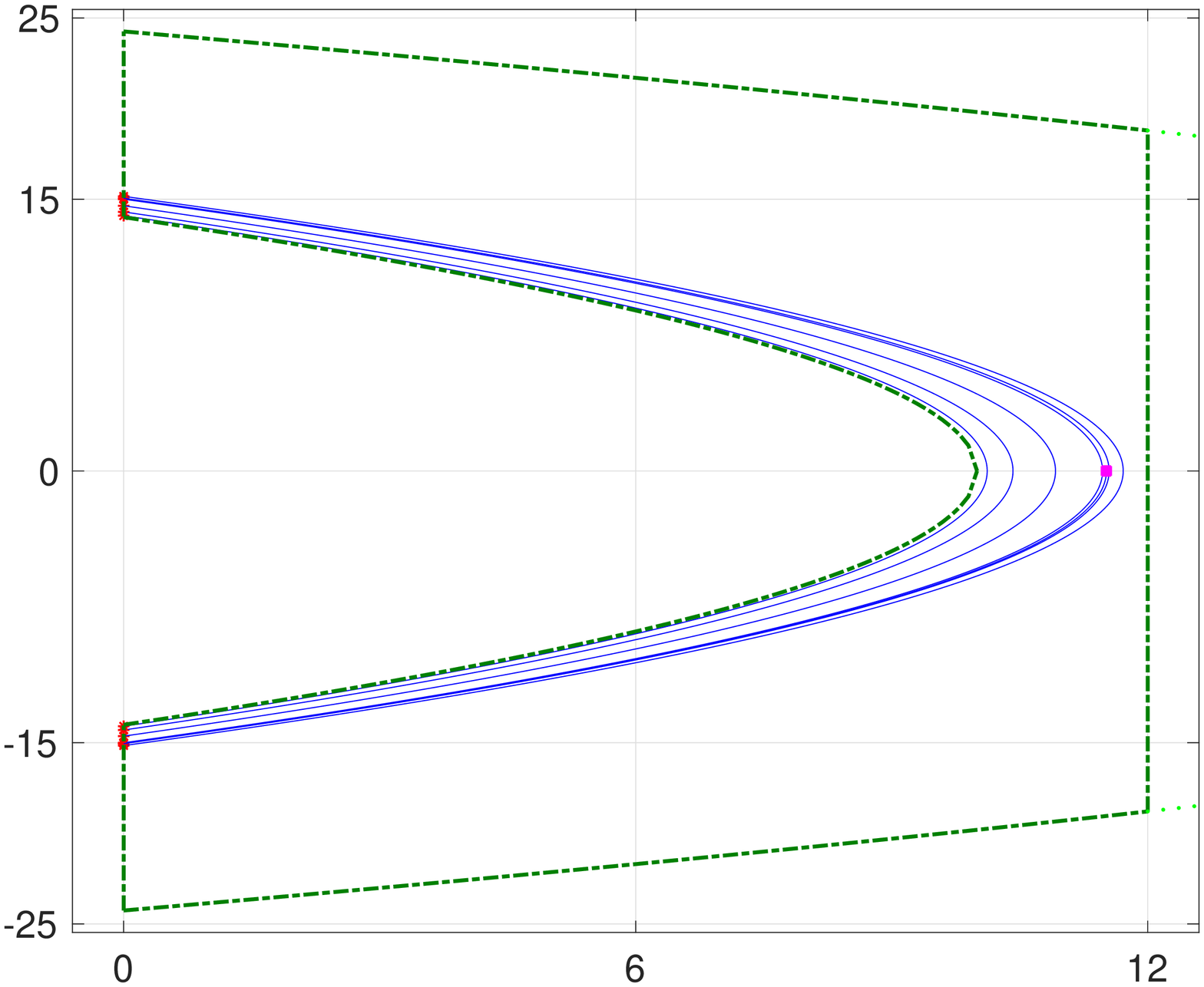}
		\put(-1.06,.43){$x_2$}
		\put(-.5,-.03){$x_1$}
		\put(-.2,.2){$\K$}
		\put(-.2,.54){(11,0)}
		\put(-.14,.53){\line(1,-2){0.05}}
		\caption{Solution on the $(x_1, x_2)$ plane.}
		\label{fig:tball0}
	\end{subfigure}
	\caption{\label{fig:bball0} Simulation of $\Huw$ controlled by $\kmd$ in \eqref{eq:Brd}.}
\end{figure}
\ConfSp{-3mm}

%% file: Sections/appendix.tex
\section{Appendix}

\subsection{Definitions and Related Results}
\label{app:others}
\begin{definition}(outer semicontinuity of set-valued maps)\label{def:osc}
	A set-valued map $S: \reals^n \rightrightarrows \reals^m$ is outer semicontinuous at $x \in \reals^n$ if for each sequence $\{x_i\}^{\infty}_{i=1}$ converging to a point $x \in \reals^n$ and each sequence $y_i \in S(x_i)$ converging to a point $y$, it holds that $y \in S(x)$; see  \cite[Definition 5.4]{rockafellar2009variational}. Given a set $K \subset \reals^n$, it is outer semicontinuous relative to $K$ if the set-valued mapping from $\reals^n$ to $\reals^m$ defined by $S(x)$ for $x\in K$ and $\emptyset$ for $x \notin K$ is outer semicontinuous at each $x\in K$.
	\hfill $\square$
\end{definition}


\begin{definition}(lower semicontinuous set-valued maps)
	A set-valued map $S: \reals^n \rightrightarrows \reals^m$ is lower semicontinuous if for every $x \in \reals^n$, one has that $\liminf\limits_{x_i \rightarrow x} S(x_i) \supset S(x),$ where 
	$$\liminf\limits_{x_i \rightarrow x} S(x_i) : = \{z: \forall x_i \rightarrow x, \exists z_i \rightarrow z \text{ s.t. } z_i \in S(x_i) \}$$
	is the inner limit of $S$ (see \cite[Chapter 5.B]{rockafellar2009variational}). 
\end{definition}

\begin{lemma}(\cite[Theorem 2.9.10]{clarke1990optimization})\label{lem:innerTcone}
	Given a set $S : = \{x: h(x) \leq 0\}$, suppose that, for every $x \in \{x: h(x) = 0\}$, $h$ is directionally Lipschitz at $x$ with $0 \notin \nabla h(x) \neq \emptyset$ and the collection of vectors $Y : = \{y : \inner{\nabla h(x)}{y} <\infty \}$ is nonempty.
	Then, $S$ admits a hypertangent at $x$ and
	\begin{enumerate}[leftmargin = 4mm, label = \arabic*)]
		\item\label{item:inner1} $y \in T_S(x) \ \ \ifeq \inner{\nabla h(x)}{y} \leq 0$;
		\item\label{item:inner2} $\exists y \in \inter T_S(x)\cap \inter Y \ \ \text{s.t.}\ \  \inner{\nabla h(x)}{y} < 0$.
	\end{enumerate}
	\vspace{-6mm}
	\hfill $\triangle$
\end{lemma}

\begin{corollary}(\cite[Corollary 2 of Theorem 2.9.8]{clarke1990optimization})\label{coro:intersect}
	Let $C_1, C_2 \subset \reals^n$ and $x \in C_1 \cap C_2$.
	Suppose that
	$$T_{C_1}(x) \cap \inter T_{C_2}(x) \neq \emptyset,$$
	and that $C_2$ admits at least one hypertangent vector at $x$.
	Then, if $C_1$ and $C_2$ are regular at $x$, one has
	$$T_{C_1}(x) \cap T_{C_2}(x) = T_{C_1 \cap C_2}(x).$$
\end{corollary}


\begin{corollary}(\cite[Corollary 2.13]{robust})\label{coro:2.13}
	Given a lower semicontinuous set-valued map $W$ and an upper semicontinuous function $w$, the set-valued map defined for each $z$ as $S(z) := \{z' \in W (z) : w(z, z') < 0 \}$ is lower semicontinuous.
\end{corollary}

\begin{theorem}(Michael Selection Theorem, \cite[Theorem 2.18]{robust})\label{thm:mselection}
	Given a lower semicontinuous set-valued map $S: \reals^n \rightrightarrows \reals^m$ with nonempty, convex, and closed values, there exists a continuous selection $s : \reals^n \rightarrow \reals^m$.
\end{theorem}

\begin{theorem}(\cite[Theorem 4.1]{dugundji1951extension})\label{thm:4.1}
	Given a closed set $A \subset \reals^n$ and a continuous map $s : A \mapsto \reals^m$, there exists a continuous extension $\widetilde{s} : \reals^n \mapsto \reals^m$ of $s$.
	Furthermore, $\widetilde{s}(x) \subset \overline\co (s(A))$ for every $x \in \reals^n$.
\end{theorem}

\begin{proposition}(Minimal Selection Theorem \cite[Proposition 2.19]{robust})\label{prop:2.19}
	Let the set-valued map $S: \reals^n \rightrightarrows \reals^m$ be lower semicontinuous with closed graph and nonempty closed convex values.
	Then, the minimal selection $m : \reals^n \rightarrow \reals^m$, which is given by
	$$m(x) := \argmin \left\{ |z|: z \in S(x)\right\},$$
	is locally bounded, and if $\gph m$ is closed and, then, $m(x)$ is continuous.
\end{proposition}


\begin{lemma}\label{lem:Tcone}
	Consider a closed set $\Ccl \subset \reals^n \times \Wc$ that has $0 \in \wxc(x)$ for every $x \in \Pwc{\Ccl}$ and a map $\Fcl : \reals^n \times \Wc \rightrightarrows \reals^n$ satisfying item~\ref{item:A2} in Definition~\ref{def:hbc}.
	Suppose there exists a pair $(\Ly, \rU)$, where the continuous function $\Ly$ is continuously differentiable on an open set containing $\Lss$ and $\rU \in \reals$ is such that, for some $\level < \rU$, items \eqref{eq:ly1} and \ref{item:Ly1}-\ref{item:Ly3} hold.
	Then, for every $x \in \partial(\K \cap \Pwc{\Cw}) \setminus \Pwd{\Dcl}$,
	\begin{align}\label{eq:Tcone}
		\Fcl(x, 0) \cap T_{\K \cap \Pwc{\Ccl}}(x) \neq \emptyset.
	\end{align}
\end{lemma}
\begin{proof}
	Let $\level < \rU$ satisfy the properties in the statement of the claim.
	Let $K_1 = \inter (\Ls )\cap \partial \Pwc{\Ccl}$,
	$K_2 = \Ly^{-1}(\level) \cap \inter (\Pwc{\Ccl})$, and
	$K_3 = \Ly^{-1}(\level) \cap \partial \Pwc{\Ccl}.$
	It is obvious that $K_1, K_2$, and $K_3$ are disjoint and $\bigcup\limits^3_{i = 1} K_i \setminus \Pwd{\Dcl} = \partial(\K \cap \Pwc{\Cw}) \setminus \Pwd{\Dcl}.$
	We have the following three cases:
	\begin{enumerate}[label = \roman*), leftmargin = 1.5\parindent]
		\item \textit{For every $x \in K_1\setminus \Pwd{\Dcl}$,}
		since $T_{\K \cap \Pwc{\Ccl}}(x) = T_{\Pwc{\Ccl}}(x)$, item \ref{item:Ly2} implies \eqref{eq:Tcone}.
		\item \textit{For every $x \in K_2\setminus \Pwd{\Dcl}$,}
		we have $T_{\K \cap \Pwc{\Ccl}}(x) = T_{\Ls}(x)$.
		Applying item~\ref{item:inner1} of Lemma~\ref{lem:innerTcone} to every such $x$ with $h(x) = \Ly(x) - \level,$ hence, $S = \Ls$, and with any point in $\Fcl(x,w_c)$ playing the role of $y$, we have that \eqref{eq:ly1} and item~\ref{item:Ly1} imply $\Fcl(x,w_c) \subset T_{\Ls}(x)$ for every $w_c \in \uxc(x)$.
		Then, with the assumption that $0 \in \uxc(x)$ for every $x \in \Pwc{\Cw}$, \eqref{eq:Tcone} holds.
		\item \textit{For every $x \in K_3\setminus \Pwd{\Dcl}$,}
		we argue that there exists a vector $\xi \in \Fcl(x, 0) \cap T_{\Pwc{\Ccl}}(x)$ that is also contained in $T_{\Ls \cap \Pwc{\Ccl}}(x).$
		To this end, for every $x \in K_3\setminus \Pwd{\Dcl}$, consider $\xi \in \Xi_x$ as defined in \ref{item:Ly3}.
		For a given $x \in K_3\setminus \Pwd{\Dcl}$, let
		$$\widetilde{C}_x: = \{x + \alpha \xi: \alpha \geq 0 \} \cap \Pwc{\Ccl}.$$
		If $\widetilde{C}_x = \{x\}$, we have $\xi = 0$ by the fact that $x \in K_3 \subset \Pwc{\Ccl}$ and item~\ref{item:Ly2}, which contradicts with item~\ref{item:Ly3}.
		Hence, for every such $x$, $\widetilde{C}_x$ has more than one point and $\xi \neq 0$.
		Then, there exists $x' \neq x$ such that $x' = (\alpha' \xi + x) \in \widetilde{C}_x$.
		By definition of $\widetilde{C}_x$, for each $\lambda \in [0,1]$, $x'' = \lambda x + (1 -\lambda) x'$ is also in $\widetilde{C}_x$.
		Let $C_x = \con\{x,x'\}$.
		By construction, $C_x$ is a convex subset of $\widetilde{C}_x$ and is not a singleton.
		Next, for every $x \in K_3 \setminus \Pwd{\Dcl}$, we apply Corollary~\ref{coro:intersect} with $C_1 = C_x$ and $C_2 = \Ls$.
		Item \ref{item:Ly3} implies $T_{C_x}(x) \cap \inter T_{\Ls}(x) \neq \emptyset$.
		Applying Lemma~\ref{lem:innerTcone} with $h(x) = \Ly(x) - \level$, the set $\Ls$ admits a hypertangent at every $x \in \Ly^{-1}(\level) \cap \Pwc{\Cw}$.
		Then, \cite[Corollary 2 of Theorem 2.4.7 (page 56)]{clarke1990optimization} implies the set $\Ls$ is regular at every $x$ with $f(x) = \Ly(x) - \level$.
		Since set $C_x$ is regular at $x$ by construction, Corollary~\ref{coro:intersect} implies that for every $x \in K_3 \setminus \Pwd{\Dcl}$,
		$$T_{C_x}(x) \cap T_{\Ls}(x) = T_{\Ls \cap C_x}(x).$$
		Because of the properties of tangent cones in \cite[Table 4.3, item (1)]{aubin2009set} and the fact that $C_x \cap \Ls \subset \Pwc{\Ccl} \cap \Ls$ by construction of $C_x$, we also have
		$$T_{\Ls \cap C_x}(x) \subset T_{\Ls \cap \Pwc{\Ccl}}(x).$$
		Then, by definition of tangent cone, $\xi \in T_{C_x}(x)$ and
		$\xi \in (T_{\Ls}(x) \cap T_{C_x}(x)) \subset T_{\Ls \cap \Pwc{\Ccl}}(x).$
		Therefore, by assumption, since $\xi \in \Fcl(x, 0) \cap T_{\Pwc{\Ccl}}(x)$ and the fact that $\K \cap \Pwc{\Ccl} = \Ls \cap \Pwc{\Ccl},$ \eqref{eq:Tcone} holds for every $x \in K_3 \setminus \Pwd{\Dcl}$.
	\end{enumerate}

\ConfSp{-4mm}
\end{proof}




\NotConf{
\subsection{Proof of Theorem~\ref{prop:Ly_pre}}
\label{appendix:ProofOfprop:Ly_pre}
	Consider the $\Lss$ restriction to the hybrid system $\Hcl$, denoted $\wHS$ and whose data is $(\wC, \Fcl,\wD$,$\Gcl)$, where the flow set and the jump set are given by $\wC = (\Lss \times \Wc) \cap \Ccl$ and $\wD = (\Lss \times \Wd) \cap \Dcl$, respectively.
	Fix $\level \in (-\infty,\rU)$ such that \eqref{eq:ly1}, \eqref{eq:ly2}, and \eqref{eq:lyJump} hold and $\K$ is nonempty and closed.
	For any nontrivial\footnote{Trivial solution pairs always stay within the set of interest.} $(\phi, w) \in \sol_{\wHS}(\K)$, pick any $(t , j) \in \dom \phi$ and let $0 = t_0 \leq t_1 \leq t_2 \leq ... \leq t_{j+1} = t$ satisfy
	$$\dom \phi \cap ([0, t] \times \{0,1, ..., j\})= \bigcup\limits^j_{k = 0} \left([t_k, t_{k + 1}] \times \{k\}\right).$$
	Next, we show that $\rge \phi \subset \Ls$.
	Proceeding by contradiction, suppose there exists $(t^*,j^*) \in \dom \phi$ that $\phi(t^*,j^*) \in \Lss \setminus \Ls$, i.e.,
	\begin{align}\label{eq:level}
		\level < \Ly(\phi(t^*,j^*)) \leq \rU.
	\end{align}
	Without lose of generality, we have the following two cases:
	\begin{enumerate}[label = \roman*), leftmargin=1.5\parindent]
		\item $\phi$ leaves $\Ls$ by ``jumping'' at $(t^*, j^*)$:\\
		namely, $\phi(t,j) \in \K$ for all $(t,j) \in \dom \phi$ with $t + j < t^*+ j^*$, and $(\phi(t^*,j^*-1), w_d(t^*,j^*-1)) \in (\Ls \times \Wd) \cap \Dcl$.
		Hence, using \eqref{eq:ly2}, it implies $\Ly(\phi(t^*, j^*)) \leq \level$, which contradicts \eqref{eq:level};
		\item $\phi$ leaves $\Ls$ by ``flowing'' during the interval $I^{j^*} : = [t_{j^*}, t_{j^* + 1}]$:\\
		due to absolute continuity of $t \mapsto \phi (t, j)$ on $I^{j^*}$, $\phi$ leaves $\Ls \cap \Pwc{\Ccl}$ and enters $(\Lss \setminus \Ls) \cap \overline{\Pwc{\Ccl}}$.
		More precisely, since $\Ls \subsetneq \Lss$, by closedness of $\Ls$, there exists a hybrid time instant $(\tau^*, j^*) \in \dom \phi$ such that $(\phi(\tau^*,j^*), w_c(\tau^*,j^*))$ $\in (\Ly^{-1}(\level) \times \Wc) \cap \Ccl$ and $(\phi(t, j^*), w_c(t, j^*)) \in ((\Lss \setminus \Ls) \times \Wc) \cap \overline{\Ccl}$ for all $t \in (\tau^*, t^*]$, where $t_{j^*} < \tau^* < t^*\leq t_{j^*+ 1}$.
		Moreover, by \IfConf{item (S1$_w$) in \cite[Definition 2.1]{PartI-Arxiv}}{item~\ref{item:Sw1} in Definition~\ref{def:solution}}, for every $t \in \inter I^{j^*}$, $(\phi(t,j^*), w_c(t,j^*)) \in \wC$.
		Then, \eqref{eq:ly1} implies that for almost all $t \in [\tau^*, t^*]$,
		$$\frac{d}{dt} \Ly(\phi(t,j^*)) \leq 0.$$
		Integrating both sides, we have
		$$\Ly(\phi(t^*, j^*)) \leq \Ly(\phi(\tau^*, j^*)),$$
		which leads to $\Ly(\phi(t^*,j^*)) \leq \Ly(\phi(\tau^*, j^*)) = \level$.
		This contradicts \eqref{eq:level}.
	\end{enumerate}
	
	Next, we establish robust forward pre-invariance of $\K$ for $\wHS$ when \eqref{eq:lyJump} holds.
	By \IfConf{item (S1$_w$) in \cite[Definition 2.1]{PartI-Arxiv}}{item~\ref{item:Sw1} in Definition~\ref{def:solution}} and closedness of $\K$, every $(\phi, w) \in \sol_{\wHS}(\K)$ stays within $\K$ during flow.
	Therefore, if $\phi$ leaves $\K$ and enters $\Ls \setminus \K$, it must have jumped.
	Suppose there exists $(\phi, w) \in \sol_{\wHS}(\K)$ that has its $\phi$ element left $\K$ eventually, while \eqref{eq:lyJump} holds.
	Then, for every such $(\phi, w)$, there exists $(t^*, j^*) \in \dom \phi$ such that $\phi(t^*, j^*) \in \Ls \setminus (\Pwc{\Ccl} \cup \Pwd{\Dcl})$ and $(\phi(t^*, j^* - 1), w_d(t^*, j^* - 1))\in (\K\times \Wd) \cap \Dcl$.
	This leads to a contradiction with \eqref{eq:lyJump}.
	Thus, $\K$ is robustly forward pre-invariant for $\wHS$.
	
	To complete the proof, we show that every $(\phi, w) \in \sol_{\wHS}(\K)$ with $\rge \phi \subset \K$ is also a maximal solution to $\Hcl$.
	Proceeding by contradiction, suppose there exists $(\phi, w) \in \sol_{\wHS}(\K)$ with $\rge \phi \subset \K$ that can be extended outside of $\K$ for $\Hcl$.
	More precisely, there exists $(\psi, v) \in \sol_{\Hcl}(\K)$, such that $\dom \psi \setminus \dom \phi \neq \emptyset$, for every $(t,j) \in \dom \phi, (\psi(t,j), v(t,j))= (\phi(t,j), w(t,j))$ and for every $(t,j) \in \dom \psi \setminus \dom \phi$, $\psi(t,j) \notin \K$.
	Let $(T,J) = \sup \dom \phi$. We have two cases:
	\begin{enumerate}[label = \roman*), resume]
		\item $(\psi, v)$ extends $(\phi, w)$ via flowing:\\
		namely, $(\psi(T,J), v_c(T,J)) = (\phi(T,J)$,  $w_c(T,J)) \in (\K \times \Wc) \cap \Ccl$, $t \mapsto \psi(t,J)$ is absolute continuous on $I^J$.
		By \IfConf{item (S1$_w$) in \cite[Definition 2.1]{PartI-Arxiv}}{item~\ref{item:Sw1} in Definition~\ref{def:solution}}, $(\psi(t,J), v_c(t,J))\in \overline{\Ccl}$ for all $t \in \inter I^J$.
		Thus, it must be the case that $\psi(t,J) \in \overline{\Pwc{\Ccl}}\setminus \Ls$ for some $t \in I^J$.
		Since $\Ls \subsetneq \Lss$, there exists $t^* \in I^J$ such that $\psi(t^*,J) \in \Lss \cap ( \overline{\Pwc{\Ccl}}\setminus \Ls)$, which is an extension of $(\phi, w)$ for $\wHS$.
		This contradicts with the maximality of $(\phi, w)$ to $\wHS$.
		\item $(\psi, v)$ extends $(\phi, w)$ via jumping:\\
		namely, $(\psi(T,J), v_d(T,J)) = (\phi(T,J)$,  $w_d(T,J)) \in (\K \times \Wd) \cap \Dcl$ and $\psi(T, J+1) \notin \K$.
		By \IfConf{item (S2$_w$) in \cite[Definition 2.1]{PartI-Arxiv}}{item~\ref{item:Sw2} in Definition~\ref{def:solution}}, this contradicts with the maximality of $(\phi, w)$ to $\wHS$.
	\end{enumerate}
}

\ConfSp{-3mm}
\subsection{Proof of Theorem~\ref{prop:Ly}}
\label{appendix:ProofOfprop:Ly}
	First, applying \IfConf{\cite[Proposition 3.4]{PartI-Arxiv}}{Proposition~\ref{prop:wexistence}}, there exists a nontrivial solution pair to $\Hcl$ from every $x \in \K$.
	Then, it follows from Theorem~\ref{prop:Ly_pre} that $\K$ is robustly forward pre-invariant for $\Hcl$.
	Such a property implies that every maximal solution pair  $(\phi, w)$ to $\Hcl$ from $\K$ has $\rge \phi \subset \K$.
	Next, we show by applying \IfConf{\cite[Proposition 3.4]{PartI-Arxiv}}{Proposition~\ref{prop:wexistence}} that every maximal solution pair $(\phi, w)$ to $\Hcl$ starting from $\K$ is also complete.
	Case \IfConf{b.1.1) in \cite[Proposition 3.4]{PartI-Arxiv}}{\ref{item:b.1.1} in Proposition~\ref{prop:wexistence}} is excluded for every $(\phi, w) \in \sol_{\Hw}(\K)$ since $\K \cap \Pwc{\Ccl}$ is closed.
	Cases \IfConf{b.1.2) and c.2)}{\ref{item:b.1.2} and \ref{item:c.2}} are excluded since \eqref{eq:Tcone} holds for every $x \in \K \setminus \Pwd{\Dw}$.
	This follows from Lemma~\ref{lem:Tcone}, and the fact that $\K \subset \Pwc{\Cw} \cup \Pwd{\Dw}$ and $T_{\Ls\cap \Pwc{\Cw}}(x) = \reals^n$ for every $x \in \inter (\Ls\cap \Pwc{\Cw}).$
	Case \IfConf{b.2)}{\ref{item:b.2}} is not possible for every maximal solution from $\K$ by assumption \ref{item:Ly4}.
	Finally, when \eqref{eq:lyJump} holds, namely, $\Gcl((\K\times \Wd) \cap \Dcl) \subset \K$, case \IfConf{c.1)}{\ref{item:c.1}} in \IfConf{\cite[Proposition 3.4]{PartI-Arxiv}}{Proposition~\ref{prop:wexistence}} does not hold.
	Therefore, only case \IfConf{a)}{\ref{item:a}} is true for every maximal solution pair starting from $\K$.
	Hence $\K$ is robustly forward invariant for $\Hcl$.